\documentclass[11pt]{article}
\usepackage{amsmath,amssymb,amsfonts,amsthm}
\usepackage{mathrsfs}
\usepackage{indentfirst}
\usepackage[pdftex]{color,graphicx}
\usepackage[pdftex,bookmarks,unicode,colorlinks]{hyperref}
\usepackage{enumerate} 
\usepackage[numbers]{natbib} 
\usepackage{yfonts} 
\usepackage{stmaryrd} 
\usepackage{extarrows}
\usepackage{longtable}
\usepackage{caption} 
\usepackage[toc,page]{appendix}
\usepackage{mathtools}
\usepackage[ruled,vlined,linesnumbered]{algorithm2e}
\usepackage{subcaption}
\usepackage{refcheck}

\usepackage{tikz-cd}
\usetikzlibrary{matrix,arrows,decorations.pathmorphing}

\allowdisplaybreaks[4]

\theoremstyle{plain}
\newtheorem{theorem}{Theorem}[section]
\newtheorem{proposition}[theorem]{Proposition}
\newtheorem{corollary}[theorem]{Corollary}
\newtheorem{lemma}[theorem]{Lemma}

\theoremstyle{remark}
\newtheorem{remark}[theorem]{Remark}

\theoremstyle{definition}
\newtheorem{definition}[theorem]{Definition}

\newcommand{\R}{\mathbb R}

\newcommand{\e}{\epsilon}

\newcommand{\Ad}{\operatorname{Ad}}

\newcommand{\pt}{\partial}

\newcommand{\GL}{\mathrm{GL}}

\newcommand{\ad}{\operatorname{ad}}
\newcommand{\varg}{\textsl{g}}

\DeclareFontFamily{U}{mathx}{\hyphenchar\font45}
\DeclareFontShape{U}{mathx}{m}{n}{
      <5> <6> <7> <8> <9> <10>
      <10.95> <12> <14.4> <17.28> <20.74> <24.88>
      mathx10
      }{}
\DeclareSymbolFont{mathx}{U}{mathx}{m}{n}
\DeclareFontSubstitution{U}{mathx}{m}{n}
\DeclareMathAccent{\widecheck}{0}{mathx}{"71}
\DeclareMathAccent{\wideparen}{0}{mathx}{"75}

\makeatletter
\def\widebreve#1{\mathop{\vbox{\m@th\ialign{##\crcr\noalign{\kern3\p@}%
      \brevefill\crcr\noalign{\kern3\p@\nointerlineskip}%
      $\hfil\displaystyle{#1}\hfil$\crcr}}}\limits}

\def\brevefill{$\m@th \setbox\z@\hbox{$\braceld$}%
  \bracelu\leaders\vrule \@height\ht\z@ \@depth\z@\hfill\braceru$}
\makeatletter

\numberwithin{equation}{section} 

\title{\bf Invariant connections in geometric mechanics: \\
reduction, nonlocality, and curvature effects}

\author{
\normalsize{
Qiao Huang\footnote{School of Mathematics, Southeast University, Nanjing 211189, P.R. China. Email: \texttt{qiao.huang@seu.edu.cn}}
}
}
\date{}

\begin{document}

\maketitle
\vspace{-0.4in}

\begin{abstract}
We present a systematic study on the role of invariant connections in geometric mechanics. We first develop a comprehensive theory of reduction under left- and right-invariant connections on Lie groups, showing that the Euler--Poincar\'e and Lie--Poisson equations are independent of the connection. We then introduce a novel connection-dependent variational principle, where the Lagrangian depends on the velocity parallel-transported back to the initial point of the curve. For Cartan--Schouten connections, this leads to an integro-differential Euler--Poincar\'e equation that exhibits two distinct sources of nonlocality in time: a path-dependent term encoded in the parallel transport and a future-dependent term arising from a curvature integral. We reformulate this equation as a two-point boundary value problem and present a two-level numerical scheme. The general theory is illustrated on the Heisenberg group, where the equations simplify due to nilpotency, and on the rotation group, where the full integro-differential structure is retained.
\end{abstract}

\section{Introduction}

Geometric mechanics is the study of mechanical systems through the lens of differential geometry \cite{AM78,LM87}, with applications spanning from classical particle \cite{DR89} and rigid body \cite{HSS09} to Cosserat continuum \cite{ED98}, fluid mechanics \cite{KW09}, control theory \cite{BL04}, thermodynamics \cite{GY17,HZ23b}, and field theory \cite{DR85}, etc. The interplay between symmetry and reduction lies at the heart of modern geometric mechanics. The observation that the Euler--Lagrange and Hamilton equations on a Lie group descend to equations on its Lie algebra and dual is one of the cornerstones of the subject. The Euler--Poincar\'e and Lie--Poisson reduction theorems, originally developed for left-invariant systems, provide a systematic framework for exploiting continuous symmetries. These results have been extensively studied in the classical works of Arnold--Khesin \cite{Arn66,AK98,KMM21}, Marsden--Ratiu--Holm \cite{EM70,HMR98,MR99}, and others \cite{OR04,GHR11,CHR18,ACD24,Ost98}, and are now standard in the geometric mechanics literature. However, the role of the underlying affine connection in the reduction process has received comparatively little attention, despite the fact that connections arise naturally in many mechanical systems, particularly those with nonholonomic constraints \cite{Blo15} or with Lagrangians that depend explicitly on the geometry of the configuration space.

A distinguished class of connections on Lie groups is the family of Cartan--Schouten connections, which are bi-invariant connections whose geodesics through the identity coincide with one-parameter subgroups. These connections, introduced by Cartan and Schouten in the context of symmetric spaces \cite{CS26}, have been studied extensively in the differential geometry literature \cite{Nom54,Bec72,Laq92}. They are characterized by the property that their associated bilinear map on the Lie algebra is proportional to the Lie bracket \([\cdot,\cdot]_{\mathfrak g}\), and the most prominent examples are the $(\pm)$-connections, which are flat and have torsion equal to \(\pm [\cdot,\cdot]_{\mathfrak g}\), and the symmetric $(0)$-connection, which is torsion-free and locally symmetric. Cartan--Schouten connections have found applications in geometric control \cite{Bul95}, nonholonomic mechanics \cite{Bar16}, computational anatomy \cite{PL20}, and information geometry \cite{DMS24}, where they provide a natural geometric framework for studying invariant structures on Lie groups. Despite these applications, a systematic treatment of their role in reduction theory, particularly for Lagrangians that depend on the connection itself, has been lacking.

In this paper, we address this gap by providing a unified geometric framework for reduction on Lie groups equipped with left- or right-invariant connections. The main contributions of this paper are twofold. 

We first develop a comprehensive theory of reduction under left- and right-invariant connections, establishing the equivalence between the global Euler--Lagrange and Hamilton's equations and their reduced counterparts, the Euler--Poincar\'e and Lie--Poisson reduction equations, on the Lie algebra and its dual. This theory is valid for any invariant connection, and the reduction equations are shown to be independent of the connection: the connection terms cancel exactly in the reduction process, yielding the classical reduction equations. This result clarifies the geometric robustness of the reduction procedure and explains why the classical local-coordinate formulations contain no connection terms.

We then introduce a novel connection-dependent variational principle, in which the Lagrangian depends not on the instantaneous velocity \(\dot \gamma(t)\), but on the velocity \(v(t)=\Gamma(g)_t^0\dot \gamma(t)\) parallel-transported back to the initial point of the curve. This formulation is natural when one wishes to work in a fixed reference vector space and to separate the dynamical variables from the geometry of the connection. 
For Cartan--Schouten connections, we derive an integro-differential Euler--Poincar\'e equation that exhibits two distinct sources of nonlocality when the connection is non-flat. First, the parallel transport term \(\Gamma^*(g)_t^0\frac{\delta l}{\delta g}\) encodes the holonomy of the connection along the curve, introducing a path-dependent nonlocality that reflects the non-abelian structure of the group. Second, the future integral term contributes a future-dependent correction that arises from the curvature of the connection. In the flat cases ($(\pm)$-connections), the parallel transport reduces to pure left or right translation, and the curvature integral vanishes; the entire equation then localizes to the classical left and right Euler--Poincar\'e equations. Unlike the parallel transport term, which is intrinsic to the group structure, the curvature-induced future dependence can be eliminated by introducing an auxiliary variable, reducing the future integral term to a terminal value problem and reflecting the two-point boundary value structure of the variational problem expressed in transported velocities. The interplay between these two sources of nonlocality in time is a central theme of our connection-dependent variational framework: The nonlocality of the parallel transport term is intrinsic and unavoidable, while that of the curvature integral term is eliminable.
We provide a detailed numerical analysis of the equation's structure, its reformulation as a two-point boundary value problem, and its numerical solution via a two-level scheme. The general theory is illustrated with two concrete examples: the Heisenberg group (in Appendix \ref{sec-app-H3}), where the parallel transport and the reduced equations simplify explicitly due to nilpotency, and the rotation group \(\mathrm{SO}(3)\), where the full integro-differential structure is retained and solved numerically for both the free rigid body and the harmonic oscillator.

The paper is organized as follows. Section \ref{sec-prel} reviews the necessary background on Lie groups, tangent and cotangent bundles, and horizontal differentials. Section \ref{sec-inv-conn} presents the theory of left-invariant connections on Lie groups, including their curvature, torsion, and parallel transport. Section \ref{sec-geom-mech} develops the Euler--Poincar\'e and Lie--Poisson reduction theorems for left-invariant connections, establishing the equivalence between the global Euler--Lagrange and Hamilton's equations and their reduced counterparts on the Lie algebra and its dual. Section \ref{sec-CS-conn} introduces Cartan--Schouten connections and their bi-invariance and special cases. Section \ref{sec-conn-vp} presents the connection-dependent variational principle and derives the integro-differential Euler--Poincar\'e equation. Section \ref{sec-solv} discusses the solvability of the reduced equation reformulated as a two-point boundary value problem, where we 
design a two-level scheme and apply it to the rigid body. 
Section \ref{sec-conclusion} summarizes the main results and beyond.
Finally, the appendices collect selected proofs and supplementary material supporting the theoretical developments, examples, and numerical implementation.

\section{Preliminaries}\label{sec-prel}

This section reviews the basic differential geometry that will be needed in the sequel. We fix notations and recall some standard facts about tangent and cotangent bundles, Lie groups and their Lie algebras. For more details, we refer to the classic textbooks \cite{KN63, DR89, MR99, Lee13, GQ20}.

\subsection{Lie groups and Lie algebras}

Let \(G\) be a Lie group with identity element \(e\). For each \(g\in G\), the left and right translations are the diffeomorphisms \(\mathrm{L}_g, \mathrm{R}_g:G\to G\) defined by \(\mathrm{L}_g(h)=gh\) and \(\mathrm{R}_g(h)=hg\). A vector field \(V\) on \(G\) is called \emph{left-invariant} if it is invariant under all left translations, i.e. \((\mathrm{L}_g)_*V=V\) for every \(g\in G\). Similarly, a vector field is \emph{right-invariant} if it is invariant under all right translations. For every tangent vector \(v\in T_eG\) one obtains a left-invariant vector field \(v^{\mathrm L}\) and a right-invariant vector field \(v^{\mathrm R}\) by
\[
v^{\mathrm L}|_g = d(\mathrm{L}_g)_e(v),\qquad 
v^{\mathrm R}|_g = d(\mathrm{R}_g)_e(v),\qquad g\in G.
\]

For notational simplicity, we shall adopt the following \emph{convention} for the lifted actions of \(G\) on \(TG\) and \(T^*G\) \cite[Section 6.2]{HSS09}. For a left multiplicative action, we use \(gv := d(\mathrm{L}_g)_h v\) to denote the tangent lift of left translation, and \(g\alpha := d(\mathrm{L}_{g^{-1}})^*_{gh}\alpha\) the corresponding cotangent lift. When \(G\) is a matrix Lie group, these simplify to the usual matrix multiplication \(gv\) and \((g^{-1})^\top \alpha\), respectively. Similarly, for the right multiplicative action, we use \(v g := d(\mathrm{R}_g)_h v\) and \(\alpha g := d(\mathrm{R}_{g^{-1}})^*_{hg}\alpha\).

The space \(\mathfrak g^{\mathrm L}\) (respectively \(\mathfrak g^{\mathrm R}\)) of all left-invariant (right-invariant) vector fields is a Lie subalgebra of the Lie algebra \(\mathfrak X(G)\) of smooth vector fields on \(G\). The map \(v\mapsto v^{\mathrm L}\) (resp. \(v\mapsto v^{\mathrm R}\)) is a vector space isomorphism between \(T_eG\) and \(\mathfrak g^{\mathrm L}\) (resp. \(\mathfrak g^{\mathrm R}\)). Consequently the tangent space \(\mathfrak g:=T_eG\) inherits a Lie algebra structure $[\cdot, \cdot]_{\mathfrak g}$ by defining
\begin{equation}\label{Lie-bracket-left}
[v,w]_{\mathfrak g} := [v^{\mathrm L},w^{\mathrm L}]_e ,\qquad v,w\in\mathfrak g.
\end{equation}
With this bracket, \(\mathfrak g\) is isomorphic to \(\mathfrak g^{\mathrm L}\) as a Lie algebra, while the isomorphism \(\mathfrak g\cong\mathfrak g^{\mathrm R}\) becomes an \emph{anti-isomorphism}; indeed, one has
\begin{equation}\label{Lie-bracket-right}
[v^{\mathrm R},w^{\mathrm R}] = -\,[v,w]_{\mathfrak g}^{\,\mathrm R},\qquad v,w\in\mathfrak g.
\end{equation}
Given two vector fields $V,W$ on $G$,
\begin{equation}\label{bracket-basis}
  [V, W]_g = g \left( V_g(g^{-1}W_g) - W_g(g^{-1}V_g) + [g^{-1}V_g,g^{-1}W_g]_{\mathfrak g} \right).
\end{equation}

The \emph{exponential map} \(\exp:\mathfrak g\to G\) is defined by \(\exp(v)=e_v(1)\), where \(e_v(\cdot)\) is the integral curve of the left-invariant vector field \(v^{\mathrm L}\) (equally of the right-invariant field \(v^{\mathrm R}\)) starting at \(e\). 

The \emph{adjoint representation of the Lie group} $\Ad:G\to\GL(\mathfrak g)$ is defined by $\Ad_g = d(\mathrm{L}_g)_{g^{-1}} \circ d(\mathrm{R}_{g^{-1}})_e$; its derivative at the identity gives the \emph{adjoint representation of the Lie algebra} $\ad := d(\Ad)_e :\mathfrak g\to\mathfrak{gl}(\mathfrak g)$, $\ad_v w = [v,w]_{\mathfrak g}$. 
It is standard (e.g., \cite[Theorem 8.44]{Lee13}) that \(\Ad_g\) preserves the Lie bracket: for all \(v,w\in\mathfrak g\),
\begin{equation}\label{Ad-bracket}
[\Ad_g v,\Ad_g w]_{\mathfrak g} = \Ad_g [v,w]_{\mathfrak g}.
\end{equation}
The derivative of the adjoint map along a curve in \(G\) can be expressed in terms of the adjoint operator \cite[Proposition 6.54]{HSS09}.
Let \(g(\cdot)\) be a smooth curve on $G$ and \(w\in\mathfrak g\). Then
\begin{equation}\label{der-Ad-curve-inv}
  \frac{d}{dt} \Ad_{g(t)^{-1}} w = - \ad_{g(t)^{-1} g'(t)} \Ad_{g(t)^{-1}} w = - \Ad_{g(t)^{-1}} \ad_{g'(t) g(t)^{-1}} w.
\end{equation}

\subsection{Vertical differentials and derivatives on tangent and cotangent bundles}

In the next two subsections, we recall the notions of vertical and horizontal differentials introduced in \cite[Section 7]{HZ23}, together with the associated chain rules on tangent and cotangent bundles. These tools will be essential for the variational principles by covariant calculus developed in later sections.

Let $M$ be a smooth manifold of dimension $m$.

\begin{definition}[Vertical differentials on \(TM\)]
Let \(F\in C^\infty(TM)\). 
The \emph{vertical differential} of \(F\) is the semi-basic 1-form on \(TM\) defined by
\begin{equation}\label{eq:vert-diff}
d_{\dot x}F := i_J dF = J^*(dF) = \pt_{\dot x^i} F\,dx^i,
\end{equation}
where \(J:TTM\to TTM\) denotes the canonical almost tangent structure.
\end{definition}

Analogous constructions exist on the cotangent bundle.

\begin{definition}[Vertical derivatives on \(T^*M\)]
Let \(F\in C^\infty(T^*M)\).
The \emph{vertical derivative} of \(F\) is the vector field along the projection \(\tau^*_M:T^*M\to M\) defined by
\begin{equation}\label{eq:vert-grad}
\pt_p F := \pt_{p_i} F\,\pt_{x^i}.
\end{equation}
\end{definition}

\begin{remark}
In \cite[Subsection 7.4.2]{HZ23}, the symbol \(\nabla_p F\) was used to denote the vertical derivative defined in \eqref{eq:vert-grad}, which they termed the vertical gradient. Here we have redefined and renamed this object to better align with standard geometric conventions, reserving \(\nabla\) exclusively for covariant derivatives associated with a connection.
\end{remark}

Both \(d_{\dot x}F\) and \(\pt_p F\) are well-defined tensorially, i.e. they do not depend on the choice of local coordinates. This fact is verified by a direct change of canonical coordinates $(x,\dot x)$ and $(x,p)$.

\subsection{Horizontal differentials}

\subsubsection{On fiber bundles}

The notion of a connection can be defined on an arbitrary fiber bundle. Let \(\pi:E\to M\) be a smooth fiber bundle over a manifold \(M\). An \emph{Ehresmann connection} on \(E\) is a smooth distribution \(H\subset TE\) that is complementary to the vertical bundle \(V\pi=\ker \pi_*\). Equivalently, it is given by a bundle morphism \(\Gamma: \pi^*(TM) \to TE\) over \(E\) that splits the exact sequence
\[
0 \longrightarrow V\pi \longrightarrow TE \longrightarrow \pi^*(TM) \longrightarrow 0.
\]
In other words, \(\Gamma\) is a \(TE\)-valued 1-form on \(E\) that is horizontal over \(M\) and satisfies \(\pi_*\circ\Gamma = \operatorname{id}_{\pi^*(TM)}\). The map \(\Gamma\) is called the \emph{connection form}. 

In local coordinates \((x^i,u^\alpha)\) on \(E\), where \((x^i)\) are coordinates on the base and \((u^\alpha)\) coordinates in the fibers, the connection form can be written as
\begin{equation*}
\Gamma = dx^i \otimes \left( \pt_{x^i} - \Gamma_i^\alpha \pt_{u^\alpha} \right),
\end{equation*}
where the functions \(\Gamma_i^\alpha(x,u)\) are called the \emph{connection coefficients}. 
The horizontal subspace $H_e$ at \(e\in E\) is spanned by the vectors
\[
H_i|_e =\Gamma(\pt_{x^i}|_e) = \pt_{x^i}|_e - \Gamma_i^\alpha(e)\pt_{u^\alpha}|_e,\qquad i=1,\dots,m.
\]

For a tangent vector \(v=v^i\pt_{x^i}|_x\in T_xM\), its \emph{horizontal lift} at a point \(e\in E_x\) is
\begin{equation}\label{eq:hrz-lift-vec}
\bar v = \Gamma(v) = v^i\left( \pt_{x^i} - \Gamma_i^\alpha \pt_{u^\alpha} \right)|_e \in H_e.
\end{equation}
For a smooth curve \(\gamma\) on $M$, a curve \(\bar\gamma\) on $E$ is called a \emph{horizontal lift} of \(\gamma\), if \(\pi\circ\bar\gamma=\gamma\) and the velocity vector \(\bar\gamma'(t)\) belongs to the horizontal subspace \(H_{\bar\gamma(t)}\) for every \(t\). By the theory of ordinary differential equations, such a horizontal lift through a point \(\bar\gamma(s) = e\in E_{\gamma(s)}\) at time $s$ exists and is unique. This uniqueness implies immediately that, the velocity of \(\bar\gamma\) coincides with the horizontal lift of the velocity of \(\gamma\):
\begin{equation}\label{eq:hrz-lift-curve}
\overline{\gamma'} = \bar\gamma'.
\end{equation}
The map from $E_{\gamma(s)}$ to $E_{\gamma(t)}$ that maps $\bar\gamma(s)$ to $\bar\gamma(t)$ is called the parallel transport along \(\gamma\) from time \(s\) to \(t\), and denoted by \(\Gamma(\gamma)_s^t\).

Using the connection, we can define a horizontal part of the differential of functions on the total space. This generalizes the corresponding notion in \cite[Section 7]{HZ23} from tangent and cotangent bundles to general fiber bundles.

\begin{definition}[Horizontal differential on a fiber bundle]
Let \(F\in C^\infty(E)\). The \emph{horizontal differential} of \(F\) (with respect to Ehresmann connection \(\Gamma\)) is the semi-basic 1-form on \(E\) defined by
\begin{equation}\label{eq:hor-diff-local}
d_x^\Gamma F := \Gamma^*(dF) = \left( \pt_{x^i} F - \Gamma_i^\alpha \pt_{u^\alpha} F \right) dx^i,
\end{equation}
where $\Gamma^*$ is the dual of the bundle morphism \(\Gamma\).
\end{definition}

The horizontal differential has the following geometric interpretations.

The next lemma evaluates the horizontal differential by differentiating \(F\) along parallel transport; its proof is given in Appendix~\ref{sec-app-selected-proofs}.
\begin{lemma}[Evaluation of horizontal differential]\label{lem:horizontal-differential}
Let \(F\in C^\infty(E)\). Let \(\gamma\) be a smooth curve on $M$. Then for any $t$ and $e\in E_{\gamma(t)}$,
\begin{equation}\label{hor-diff-curve}
d_x^\Gamma F|_e (\gamma'(t)) = \frac{d}{d\e}\bigg|_{\e=0} F\bigl(\Gamma(\gamma)_t^{t+\e} e\bigr).
\end{equation}
\end{lemma}

\subsubsection{On tangent and cotangent bundles}

Given a linear connection \(\nabla\) on \(M\), the above constructions reduce to horizontal differentials on the tangent and cotangent bundles when \(E=TM\) and $T^*M$, respectively, and \(\Gamma\) is the corresponding linear connection induced by the linear connection \(\nabla\) on \(M\). In these cases, the connection coefficients are \(\Gamma_i^k(x,\dot x) = \Gamma_{ij}^k(x) \dot x^j\) and \(\Gamma_{ij}(x,p) = \Gamma_{ij}^k(x) p_k\), where \(\Gamma_{ij}^k\) are the Christoffel symbols of \(\nabla\).

For \(F\in C^\infty(TM)\), the \emph{horizontal differential} of \(F\) (with respect to \(\nabla\)) \eqref{eq:hor-diff-local}, now denoted by $d_x^\nabla F$, is the 1-form on \(TM\):
\begin{equation}\label{eq:hor-diff-1}
d_x^\nabla F = 
\left(\pt_{x^i} F - \Gamma_{ij}^k\dot x^j\pt_{\dot x^k} F\right)dx^i.
\end{equation}
Similarly, for \(F\in C^\infty(T^*M)\), the \emph{horizontal differential} of \(F\) is the 1-form on \(T^*M\) denoted and given by
\begin{equation}\label{eq:hor-diff-2}
d_x^\nabla F =
\left(\pt_{x^i} F + \Gamma_{ij}^k p_k\pt_{p_j} F\right)dx^i.
\end{equation}

Using the vertical and horizontal differentials defined above, we have the following Leibniz-type rules for functions on $TM$ composed with vector fields, and on $T^*M$ composed with 1-forms.

\begin{lemma}[Chain rules on \(TM\) and \(T^*M\) {\cite[Section 7]{HZ23}}] \quad \\
(i) Let \(V\) be a vector field on \(M\) and \(F\in C^\infty(TM)\). Denote by \(F\circ V\) the function on \(M\) obtained by restricting \(F\) to the image of \(V\). Then
\begin{equation}\label{chain-rule}
d(F\circ V) = d_x^\nabla F\circ V + (d_{\dot x}F\circ V)(\nabla_{\partial_i}V)\,dx^i.
\end{equation}
(ii) Let \(\eta\) be a 1-form on \(M\) and \(F\in C^\infty(T^*M)\). Denote by \(F\circ\eta\) the function on \(M\) obtained by restricting \(F\) to the image of \(\eta\). Then
\begin{equation}\label{chain-rule-2}
d(F\circ\eta) = d_x^\nabla F\circ\eta + \bigl(\nabla_{\partial_i}\eta\bigr)\bigl(\pt_p F\circ\eta\bigr)\,dx^i.
\end{equation}
\end{lemma}

\begin{remark}
Formula \eqref{chain-rule-2} corrects the corresponding formula in \cite[Subsection 7.4.2]{HZ23}, where the second term at the r.h.s. of the formula was mistakenly derived as $\nabla_{(\pt_p F\circ\eta)}\eta$.
\end{remark}

These vertical/horizontal differential operators play a crucial role when writing variational principles on tangent or cotangent bundles, as will be seen in Section \ref{sec-geom-mech}.

\section{Invariant connections on Lie groups}\label{sec-inv-conn}

In preparation for the study of geometric reductions in later sections, we recall the general setup of left- and right-invariant connections on a Lie group. These are connections that are invariant under all left or right translations, respectively. Their structure is completely encoded by a bilinear map on the Lie algebra. 
Since the right-invariant case is entirely analogous to the left-invariant one, in this section we present only the left-invariant formalism. The corresponding right-invariant results are included in Appendix~\ref{sec-app-right-conn} for completeness.

\begin{definition}
A connection \(\nabla\) on a Lie group \(G\) is called \emph{left-invariant} if for every pair of left-invariant vector fields \(V,W\), the covariant derivative \(\nabla_V W\) is again left-invariant. Equivalently, left translations are affine isomorphisms: \((\mathrm{L}_g)^*\nabla = \nabla\) for all \(g\in G\).
\end{definition}

For a left-invariant connection $\nabla$, the covariant derivative of left-invariant vector fields is determined by a bilinear map \(\nabla^{\mathrm{L}}:\mathfrak g\times\mathfrak g\to\mathfrak g\), called the \emph{left bilinear map associated with $\nabla$}, defined by
\[
\nabla^{\mathrm{L}}_v w := (\nabla_{v^{\mathrm L}}w^{\mathrm L})_e,
\qquad \nabla_{v^{\mathrm L}}w^{\mathrm L}=(\nabla^{\mathrm{L}}_v w)^{\mathrm L}.
\]
For arbitrary vector fields \(V\) and \(W\), the covariant derivative reads
\begin{equation}\label{eq:left-conn}
  (\nabla_V W)_g = g \left[ V_g(g^{-1} W_g) + \nabla^{\mathrm{L}}_{g^{-1} V_g} (g^{-1} W_g) \right], \quad g\in G.
\end{equation}

We denote by \(\nabla^{\mathrm{L}*}:\mathfrak g\times\mathfrak g^*\to\mathfrak g^*\) the adjoint of \(\nabla^{\mathrm{L}}\) with respect to the natural pairing:
\begin{equation*}
  \nabla^{\mathrm{L}*}_v \alpha (w) = \alpha (\nabla^{\mathrm{L}}_v w), \quad \alpha \in \mathfrak g^*, v,w\in \mathfrak g.
\end{equation*}
Along a curve \(g(\cdot)\), we have the following identities.
Their proof consists of substituting the curve velocity into the left-trivialized connection formula \eqref{eq:left-conn} and then obtaining the covector identity by duality.
\begin{lemma}[Covariant derivative along a curve]\label{cov-der-curve}
Let \(g(\cdot)\) be a smooth curve on $G$. For a vector field \(W\) along \(g(\cdot)\),
\begin{equation}\label{eq:cov-curve-left}
\frac{D}{dt}W_{g(t)}
= d(\mathrm{L}_{g(t)})
\left( \frac{d}{dt} + \nabla^{\mathrm{L}}_{d(\mathrm{L}_{g(t)^{-1}})g'(t)} \right)
\left( d(\mathrm{L}_{g(t)^{-1}})W_{g(t)} \right).
\end{equation}
For a 1-form \(\eta\) along \(g(\cdot)\),
\begin{equation}\label{eq:cov-curve-left-dual}
\frac{D}{dt}\eta_{g(t)}
=
d(\mathrm{L}_{g(t)^{-1}})^*
\left( \frac{d}{dt} - \nabla^{\mathrm{L}*}_{d(\mathrm{L}_{g(t)^{-1}})g'(t)} \right)
\left( d(\mathrm{L}_{g(t)})^*\eta_{g(t)} \right),
\end{equation}
\end{lemma}

From Lemma \ref{cov-der-curve} we immediately obtain the parallel transport.

\begin{corollary}[Parallel transport]\label{cor-parr}
Along a curve \(g(\cdot)\), the parallel transport operator \(\Gamma(g)_s^t:T_{g(s)}G\to T_{g(t)}G\) is given by
\begin{equation}\label{eq:parr-left}
\Gamma(g)_s^t
=
d(\mathrm{L}_{g(t)})\,
\mathcal T\exp\left( -\int_s^t \nabla^{\mathrm{L}}_{d(\mathrm{L}_{g(r)^{-1}})g'(r)}\,dr \right)
\,d(\mathrm{L}_{g(s)^{-1}}),
\end{equation}
where $\mathcal T \exp$ is the ordered exponential.
Its dual version \(\Gamma^*(g)_s^t:T^*_{g(s)}G\to T^*_{g(t)}G\) is given by
\begin{equation}\label{eq:parr-left-dual}
\Gamma^*(g)_s^t
=
d(\mathrm{L}_{g(t)^{-1}})^*\,
\mathcal T\exp\left( \int_s^t \nabla^{\mathrm{L}*}_{d(\mathrm{L}_{g(r)^{-1}})g'(r)}\,dr \right)
\,d(\mathrm{L}_{g(s)})^*.
\end{equation}
\end{corollary}

\begin{remark}
The ordered exponential \(\mathcal T \exp\) is defined by the infinite series
\begin{equation*}
  \mathcal T \exp\left( \int_s^t A(r) dr \right)
  = \sum_{n=0}^\infty \int_s^t \int_s^{r_1} \cdots \int_s^{r_{n-1}}
  A(r_1) A(r_2) \cdots A(r_n) \, dr_n \cdots dr_2 dr_1,
\end{equation*}
where \(A(\cdot)\) is a time-dependent linear operator on a finite-dimensional vector space (or, more generally, an endomorphism-valued function). This operator is the unique solution of the initial value problem
\begin{equation*}
  \frac{d}{dt} Y(t) = A(t) Y(t), \qquad Y(s) = I,
\end{equation*}
in the sense that \(Y(t) = \mathcal T \exp\left( \int_s^t A(r) dr \right)\).
\end{remark}

\begin{proof}
Equation \eqref{eq:cov-curve-left} shows that \(\frac{D}{dt}W_{g(t)}=0\) is equivalent to 
\[
\left( \frac{d}{dt} + \nabla^{\mathrm{L}}_{g(t)^{-1} g'(t)} \right) g(t)^{-1} W_{g(t)} = 0.
\]
Solving for \(g(t)^{-1} W_{g(t)}\) yields 
\begin{equation*}
  g(t)^{-1} W_{g(t)} = \mathcal T \exp\left( - \int_s^t \nabla^{\mathrm{L}}_{g(r)^{-1} g'(r)} dr \right) g(s)^{-1} W_{g(s)}.
\end{equation*}
Equation \eqref{eq:parr-left} for \(\Gamma(g)_s^t\) follows. The dual version follows from \eqref{eq:cov-curve-left-dual}.
\end{proof}

The next formulas follow by inserting \eqref{eq:left-conn} into the definitions of curvature and torsion and using the left-trivialized bracket identity \eqref{bracket-basis}.
\begin{lemma}[Curvature and torsion]\label{cur-tor-left}
For any vector fields \(U,V,W\) on \(G\) and $g\in G$,
\begin{gather}
\begin{aligned}
(R(U,V)W)_g
= d(\mathrm{L}_g)\Big( &
\nabla^{\mathrm{L}}_{d(\mathrm{L}_{g^{-1}})U_g}
\nabla^{\mathrm{L}}_{d(\mathrm{L}_{g^{-1}})V_g}
\,d(\mathrm{L}_{g^{-1}})W_g - \nabla^{\mathrm{L}}_{d(\mathrm{L}_{g^{-1}})V_g}
\nabla^{\mathrm{L}}_{d(\mathrm{L}_{g^{-1}})U_g}
\,d(\mathrm{L}_{g^{-1}})W_g \\
& - \nabla^{\mathrm{L}}_{[d(\mathrm{L}_{g^{-1}})U_g,\,d(\mathrm{L}_{g^{-1}})V_g]}
\,d(\mathrm{L}_{g^{-1}})W_g \Big),
\end{aligned} 
\notag \\
\begin{aligned}
T(V,W)_g =
d(\mathrm{L}_g)\Big( &
\nabla^{\mathrm{L}}_{d(\mathrm{L}_{g^{-1}})V_g}
\,d(\mathrm{L}_{g^{-1}})W_g
- \nabla^{\mathrm{L}}_{d(\mathrm{L}_{g^{-1}})W_g}
\,d(\mathrm{L}_{g^{-1}})V_g - [d(\mathrm{L}_{g^{-1}})V_g,\,d(\mathrm{L}_{g^{-1}})W_g]_{\mathfrak g}
\Big).
\end{aligned} \label{eq:torsion-left}
\end{gather}
\end{lemma}

In particular, for any \(u,v,w\in \mathfrak g\),
\begin{gather*}
R(u^{\mathrm L},v^{\mathrm L})w^{\mathrm L}
=
\left(
\nabla^{\mathrm{L}}_u\nabla^{\mathrm{L}}_v w
- \nabla^{\mathrm{L}}_v\nabla^{\mathrm{L}}_u w
- \nabla^{\mathrm{L}}_{[u,v]_{\mathfrak g}}
w
\right)^{\mathrm L}, 
\\
T(v^{\mathrm L},w^{\mathrm L})
=
\left(
\nabla^{\mathrm{L}}_v w - \nabla^{\mathrm{L}}_w v - [v,w]_{\mathfrak g}
\right)^{\mathrm L}. 
\end{gather*}
That is, the curvature and torsion of a left-invariant connection are again left-invariant.

Finally, the interplay between the connection and the group structure is encapsulated in the following commutation relation for two-parameter families.

\begin{lemma}[Exchange of derivatives]
Let \(g_s(t)\) be a smooth two-parameter map into \(G\). 
Then
\begin{equation}\label{eq:exchange-der-left-conn}
\pt_s \left( g_s(t)^{-1} \pt_t g_s(t) \right) - \pt_t \left( g_s(t)^{-1} \pt_s g_s(t) \right) = \left[ g_s(t)^{-1} \pt_t g_s(t), g_s(t)^{-1} \pt_s g_s(t) \right]_\mathfrak g.
\end{equation}
\end{lemma}

\begin{proof}
Using \eqref{eq:cov-curve-left}, the standard identity \(\frac{D}{dt}\partial_s g - \frac{D}{ds}\partial_t g = T(\partial_s g,\partial_t g)\), and \eqref{eq:torsion-left}, one obtains
\begin{equation*}
  \begin{split}
    \text{l.h.s.} &= g_s(t)^{-1} \frac{D}{d s} \pt_t g_s(t) - \nabla^{\mathrm{L}}_{g_s(t)^{-1} \pt_s g_s(t)} \left(g_s(t)^{-1} \pt_t g_s(t)\right) \\
    &\quad\ - \left( g_s(t)^{-1} \frac{D}{d t} \pt_s g_s(t) - \nabla^{\mathrm{L}}_{g_s(t)^{-1} \pt_t g_s(t)} \left(g_s(t)^{-1} \pt_s g_s(t)\right) \right) \\
    &= g_s(t)^{-1} T\left( \pt_s g_s(t), \pt_t g_s(t) \right) - \nabla^{\mathrm{L}}_{g_s(t)^{-1} \pt_s g_s(t)} \left(g_s(t)^{-1} \pt_t g_s(t)\right) + \nabla^{\mathrm{L}}_{g_s(t)^{-1} \pt_t g_s(t)} \left(g_s(t)^{-1} \pt_s g_s(t)\right) \\
    &= \text{r.h.s.}.
  \end{split}
\end{equation*}
The proof is completed.
\end{proof}

\begin{remark}
Equation \eqref{eq:exchange-der-left-conn} is independent of the choice of left-invariant covariant derivatives, and is equivalent to the well-known Maurer--Cartan equations for the left Maurer--Cartan form.
\end{remark}

\section{Geometric mechanics}\label{sec-geom-mech}

In this section, we first formulate the global equations of motion for Lagrangian and Hamiltonian mechanics on a manifold equipped with a linear connection. We then apply the theory of invariant connections on Lie groups developed in the previous section to derive the corresponding reduced equations directly from these global equations. A notable outcome is that the reduced Euler--Poincar\'e and Lie--Poisson equations are independent of the choice of connection, which clarifies the consistency with the classical local-coordinate formulations.

\subsection{Least action principles}

Let \(Q\) be a smooth manifold equipped with a linear connection \(\nabla\) with torsion \(T\). 

\subsubsection*{Hamilton's principle}

Consider a regular Lagrangian \(L:\mathbb \mathrm{R}_+ \times TQ\to\mathbb R\) and the action functional
\[
\mathcal S[\gamma] = \int_0^T L(t,\gamma(t),\dot\gamma(t))\,dt,
\]
defined for curves \(\gamma\in C^2([0,T],Q)\) satisfying fixed endpoint conditions \(\gamma(0)=q_0,\;\gamma(T)=q_T\). An infinitesimal variation of \(\gamma\) is a one-parameter family of curves \(\{\gamma_\epsilon\}\) in $C^2([0,T],Q)$ with \(\gamma_0=\gamma\) and 
\begin{equation}\label{eq:inf-var}
  \delta\gamma(t) :=\frac{\partial}{\partial\epsilon}\bigg|_{\epsilon=0}\gamma_\epsilon(t)
\end{equation}
vanishing at the endpoints. The variation of the velocity is given by
\begin{equation}\label{eq:variation-der}
\delta\dot\gamma(t):=\frac{D}{d \e}\bigg|_{\e=0} \gamma'_\e(t) = \nabla_{\delta\gamma(t)}\gamma'(t)=\frac{D}{dt}\delta\gamma(t)-T(\gamma'(t),\delta\gamma(t)).
\end{equation}

\begin{proposition}[Hamilton's principle by connections {\cite{GSS03}}]
A curve \(\gamma\in C^2([0,T],Q)\) is a critical point of the action \(\mathcal S\) if and only if it satisfies the global Euler--Lagrange equation
\begin{equation}\label{eq:EL-1}
\frac{D}{dt} (d_{\dot q}L) + d_{\dot q}L \left( i_{\gamma'(t)}T \right) = d_q^\nabla L.
\end{equation}
Here \(d_q^\nabla L\) and \(d_{\dot q}L\) denote the horizontal and vertical differentials of \(L\) (see definitions \eqref{eq:vert-diff} and \eqref{eq:hor-diff-1}).
\end{proposition}

\begin{proof}
Compute the first variation of the action by the chain rule \eqref{chain-rule} and \eqref{eq:variation-der},
\begin{equation*}
  \begin{split}
    \delta \mathcal S[\gamma] &= \frac{d}{d\e}\bigg|_{\e=0} \mathcal S[\gamma_\e] = \int_0^T \frac{\pt}{\pt\e}\bigg|_{\e=0} L(t, \gamma_\e(t), \dot\gamma_\e(t)) dt \\
    &= \int_0^T \left[ d_q^\nabla L (\delta \gamma(t)) + d_{\dot q} L (\delta \dot\gamma(t)) \right] dt \\
    &= \int_0^T \left[ d_q^\nabla L (\delta \gamma(t)) + d_{\dot q} L \left( \frac{D}{dt} \delta \gamma(t) \right) - d_{\dot q} L (T(\gamma'(t), \delta \gamma(t))) \right] dt.
  \end{split}
\end{equation*}
Integrate by parts the term containing \(\frac{D}{dt}\delta\gamma\). Because \(\delta\gamma\) vanishes at the endpoints, the boundary terms disappear and one obtains
\[
\delta\mathcal S[\gamma]=\int_0^T \left[ d_q^\nabla L - \frac{D}{dt} (d_{\dot q} L) - d_{\dot q} L \left( i_{\gamma'(t)} T \right) \right] (\delta \gamma(t)) dt.
\]
Since \(\delta\gamma\) is arbitrary, the integrand must vanish, yielding \eqref{eq:EL-1}.
\end{proof}


\begin{remark}
  One can bundle the two terms $\nabla_{\gamma'(t)} \left( d_{\dot x} \mathrm{L}_0 \right)$ and $d_{\dot q}L \left( i_{\gamma'(t)}T \right)$ via the notion of opposite connections \cite[Section 2.7]{DR89}. The opposite connection of $\nabla$, denoted by $\widehat{\nabla}$, 
  when acting on a 1-form $\alpha$, is given by
  \(
    \widehat{\nabla}_U \alpha = \nabla_U \alpha + \alpha(i_U T).
  \)
  Thus, the global Euler--Lagrange equation \eqref{eq:EL-1} can be shorten as
  \begin{equation*}
    \frac{\widehat{D}}{dt} (d_{\dot q}L) = d_q^\nabla L,
  \end{equation*}
  where $\frac{\widehat{D}}{dt} = \frac{\pt}{\pt t} + \widehat{\nabla}_{\gamma'(t)}$.
\end{remark}

\subsubsection*{Hamilton's principle on the phase space}

On the cotangent bundle \(T^*Q\) we consider a Hamiltonian \(H:T^*Q\times \mathbb \mathrm{R}_+ \to\mathbb R\) and the phase-space action functional
\[
\overline{\mathcal S}[q(\cdot), p(\cdot)]=\int_0^T\bigl[p(t)(\dot q(t))-H(q(t),p(t),t)\bigr]dt,
\]
where \(p(\cdot)\) is a time-dependent 1-forms along a curve \(q(\cdot)\) in \(Q\), again with fixed configuration endpoints \(q(0)=q_0,\;q(T)=q_T\). The variation of \(\overline{\mathcal S}\) is computed analogously, using the horizontal differentials and vertical derivatives on \(T^*Q\) (definitions \eqref{eq:vert-grad} and \eqref{eq:hor-diff-2}).

This calculation leads to the following connection-based form of Hamilton's principle; details of the variation are given in Appendix~\ref{sec-app-selected-proofs}.
\begin{proposition}[Hamilton's principle on the phase space by connections]\label{prop:hamilton-principle-phase}
A curve \((q(\cdot),p(\cdot))\in C^2([0,T],T^*Q)\) is a critical point of the phase-space action \(\overline{\mathcal S}\) if and only if it satisfies the global Hamilton's equations
\begin{equation}\label{eq:HE-1}
\left\{
\begin{aligned}
&q'(t) = \pt_p H,\\
&\frac{D}{dt}p(t) + p(t)\left( i_{q'(t)}T \right) = - d_q^\nabla H.
\end{aligned}
\right.
\end{equation}
Here \(\pt_p H\) is the vertical derivative of \(H\) and \(d_q^\nabla H\) its horizontal differential (see definitions \eqref{eq:vert-grad} and \eqref{eq:hor-diff-2}).
\end{proposition}

Equations \eqref{eq:EL-1} and \eqref{eq:HE-1} are the natural generalizations of the Euler--Lagrange and Hamilton's equations to manifolds endowed with an arbitrary linear connection. They can be transformed into each other by the Legendre transform: $p = d_{\dot q}L, H(q,p)+L(q,\dot q) = p(\dot q)$. Regardless of the linear connection used, in any local coordinates, they reduce to the familiar forms
\[
\frac{d}{dt} \left( \frac{\pt L}{\pt \dot q} \right) = \frac{\pt L}{\pt q},\qquad
\dot q = \frac{\pt H}{\pt p}, \dot p = - \frac{\pt H}{\pt q}.
\]

For a Lie group equipped with a Cartan--Schouten connection, these equations will be further reduced to the Euler--Poincar\'e and Lie--Poisson equations, as shown in the following sections.

\subsection{Euler--Poincar\'e and Lie--Poisson reductions}\label{subsec-red}

Let \(G\) be a Lie group. In this subsection, we derive the reduced equations of motion on the Lie algebra \(\mathfrak g\) and its dual \(\mathfrak g^*\) from the global equations of Lagrangian and Hamiltonian formulations. 

Let \(L:\R_+ \times TG\to\mathbb R\) be a Lagrangian which may depend explicitly on time.
For a curve \(g(\cdot)\) on $G$, define the \emph{left-trivialized velocity}
\begin{equation*}
v_-(t) := g(t)^{-1} \dot g(t) \in \mathfrak g.
\end{equation*}
Write the Lagrangian in terms of the reduced coordinates $(g,v)$ on $G\times\mathfrak g$ as 
$$l_-: \R_+ \times G\times\mathfrak g \to \R, \quad l_-(g,v)=L(g,gv).$$
Define the partial differentials of \(l_-\) by
\begin{equation*}
  \frac{\delta l_-}{\delta g} := \frac{\pt l_-}{\pt g} dg \in T^*_g G, \quad \frac{\delta l_-}{\delta v} := \frac{\pt l_-}{\pt v} \in \mathfrak g^*.
\end{equation*}
It is clear that
\begin{equation}\label{vert-diff-Lie}
  d_{\dot g} L = d\left( \mathrm{L}_{g^{-1}} \right)^*_g \frac{\delta l_-}{\delta v}.
\end{equation}

For a curve \(g(\cdot)\) on \(G\), consider a smooth variation \(g_\epsilon(\cdot)\) of \(g(\cdot)\) with fixed endpoints and define the corresponding (left-translated) infinitesimal variation in the Lie algebra:
\begin{equation}\label{eq:w-left}
w_-(t) = g(t)^{-1}\delta g(t),\qquad \delta g(t)=\frac{\partial}{\partial\epsilon}\bigg|_{\epsilon=0}g_\epsilon(t).
\end{equation}
Then \(w_-(0)=w_-(T)=0\). The variation of the reduced velocity follows from the commutation relation \eqref{eq:exchange-der-left-conn}:
\begin{equation}\label{eq:delta-v-left}
\delta v_-(t) = \frac{\partial}{\partial\epsilon}\bigg|_{\epsilon=0}\bigl(g_\epsilon(t)^{-1}\dot g_\epsilon(t)\bigr)
= \dot w_-(t) + [v_-(t),w_-(t)]_{\mathfrak g}
= \dot w_-(t) + \ad_{v_-(t)} w_-(t).
\end{equation}

\begin{lemma}[Left Euler--Poincar\'e equation {\cite[Proposition 7.3]{HSS09}}]
A curve \(g(\cdot)\in C^2([0,T],G)\) is a critical point of the action \(\mathcal S[g(\cdot)]=\int_0^T L(t,g(t),\dot g(t))dt\) under fixed-endpoint variations if and only if the reduced variables \((g(\cdot),v_-(\cdot))\) satisfy the left Euler--Poincar\'e equation
\begin{equation}\label{eq:EP-left}
  \frac{d}{dt} \left( \frac{\delta l_-}{\delta v} \right) - \ad_{v_-(t)}^* \left( \frac{\delta l_-}{\delta v} \right) = d\left( \mathrm{L}_{g(t)} \right)^*_e \left( \frac{\delta l_-}{\delta g} \right).
\end{equation}
\end{lemma}


The relation between the Euler--Lagrange equation for \(L\) and the left Euler--Poincar\'e equation for \(l_-\) is established by expressing the covariant derivative and torsion in terms of left-translations.

\begin{theorem}[Left Euler--Poincar\'e reduction and reconstruction]\label{thm-EP}
Let \(G\) be equipped with a left-invariant connection whose associated left bilinear map is \(\nabla^{\mathrm{L}}\). The global Euler--Lagrange equation
\begin{equation}\label{EL-Lie}
\frac{D}{dt} (d_{\dot g}L) + d_{\dot g}L \left( i_{g'(t)}T \right) = d_g^\nabla L,
\end{equation}
is equivalent to the left Euler--Poincar\'e equation \eqref{eq:EP-left}, via the reconstruction equation \(g'(t) = g(t) v_-(t)\).
\end{theorem}

\begin{proof}
Using \eqref{vert-diff-Lie} and the dual covariant derivative formula \eqref{eq:cov-curve-left-dual}, we have
\[
\frac{D}{dt} (d_{\dot g} L) = d\left( \mathrm{L}_{g^{-1}} \right)^*_g \left( \frac{d}{dt} - \nabla^{\mathrm{L}*}_{v_-} \right) \left( \frac{\delta l_-}{\delta v} \right).
\]
From the torsion formula \eqref{eq:torsion-left}, the torsion term is, applied to an arbitrary vector $\delta g \in T_gG$,
\[
  \begin{split}
    d_{\dot g} L \left( i_{g'} T \right) (\delta g) = d_{\dot g} L \left( T(g',\delta g) \right) &= \frac{\delta l_-}{\delta v} (\nabla^{\mathrm{L}}_{v_-} g^{-1} \delta g - \nabla^{\mathrm{L}}_{g^{-1} \delta g} v_- - \ad_{v_-} g^{-1} \delta g) \\
    &= d\left( \mathrm{L}_{g^{-1}} \right)^*_g \left( \nabla^{\mathrm{L}*}_{v_-} - (\nabla^{\mathrm{L}} v_-)^* - \ad_{v_-}^* \right) \left(\frac{\delta l_-}{\delta v}\right) (\delta g),
  \end{split}
\]
where \((\nabla^{\mathrm{L}} v_-)^*\) denotes the dual of the linear map \(w \mapsto \nabla^{\mathrm{L}}_{w} v_-\). Thus,
\begin{equation*}
  d_{\dot g} L \left( i_{g'} T \right) = d\left( \mathrm{L}_{g^{-1}} \right)^*_g \left( \nabla^{\mathrm{L}*}_{v_-} - (\nabla^{\mathrm{L}} v_-)^* - \ad_{v_-}^* \right) \left( \frac{\delta l_-}{\delta v} \right).
\end{equation*}
Finally, from the evaluation formula of the horizontal lift \eqref{hor-diff-curve} and the parallel transport formula \eqref{eq:parr-left}, the horizontal differential of \(L\) applied to a vector $\delta g \in T_gG$ is
\[
  \begin{split}
    d_g^\nabla L|_{g'}(\delta g) = \frac{d}{d\e}\bigg|_{\e=0} L(\Gamma(g_\cdot)_0^\e g')
    &= \frac{d}{d\e}\bigg|_{\e=0} l_-\left( g_\e, \mathcal T \exp\left( - \int_0^\e \nabla^{\mathrm{L}}_{g_r^{-1} \pt_r g_r} dr \right) v_- \right) \\
    &= \frac{\delta l_-}{\delta g}\bigg|_{(g,v_-)}(\delta g) + \frac{\delta l_-}{\delta v}\bigg|_{(g,v_-)}(- \nabla^{\mathrm{L}}_{w_-} v_-) \\
    &= \frac{\delta l_-}{\delta g}\bigg|_{(g,v_-)}(\delta g) - d\left( \mathrm{L}_{g^{-1}} \right)^*_g (\nabla^{\mathrm{L}} v_-)^* \left(\frac{\delta l_-}{\delta v}\bigg|_{(g,v_-)}\right)(\delta g).
  \end{split}
\]
Hence,
\begin{equation*}
  d_g^\nabla L = \frac{\delta l_-}{\delta g} - d\left( \mathrm{L}_{g^{-1}} \right)^*_g (\nabla^{\mathrm{L}} v_-)^* \left(\frac{\delta l_-}{\delta v} \right)
\end{equation*}
Substituting these three relations into the Euler--Lagrange equation \eqref{EL-Lie} yields the left Euler--Poincar\'e equation \eqref{eq:EP-left}.
\end{proof}

Passing to the Hamiltonian description, we consider a Hamiltonian \(H:T^*G \times \R_+ \to\mathbb R\) and its reduction via left trivialization. For a curve \(g(\cdot), p(\cdot)\) on $T^*G$, define the \emph{left-trivialized momentum}
\begin{equation*}
\mu_-(t) := g(t)^{-1} p(t) = d\left( \mathrm{L}_{g(t)} \right)^*_e p(t)\in\mathfrak g^*.
\end{equation*}
The reduced Hamiltonian \(h_-: G\times\mathfrak g^* \times\R_+ \to \R\) is given in terms of the reduced coordinates $(g,\mu)$ on $G\times\mathfrak g^*$ by
\[
h_-(g,\mu) = H\bigl(g, g\mu\bigr).
\]
Define the partial derivatives of \(h_-\) by
\begin{equation*}
  \frac{\delta h_-}{\delta g} := \frac{\pt h_-}{\pt g} dg \in T^*_g G, \quad \frac{\delta h_-}{\delta \mu} := \frac{\pt h_-}{\pt \mu} \in\mathfrak g.
\end{equation*}
It is clear that
\begin{equation*}
  \pt_p H = d\left( \mathrm{L}_g \right)_e \frac{\delta h_-}{\delta \mu}.
\end{equation*}

Applying the reduced phase-space variational principle now gives the left Lie--Poisson equation; the calculation is included in Appendix~\ref{sec-app-selected-proofs}.
\begin{lemma}[Left Lie--Poisson equation]\label{lem:left-lie-poisson}
A curve \((g(\cdot),p(\cdot))\in C^2([0,T],T^*G)\) is a critical point of the phase-space action \(\overline{\mathcal S}[g(\cdot), p(\cdot)]=\int_0^T\bigl[p(t)(\dot g(t))-H(g(t),p(t),t)\bigr]dt\) under fixed-configuration-endpoint variations if and only if the reduced variables \((g(\cdot),\mu_-(\cdot))\) satisfy the left Lie--Poisson equation
\begin{equation}\label{eq:LP-left}
\dot\mu_-(t) - \ad_{\frac{\delta h_-}{\delta \mu}}^* \mu_-(t) = - d\left( \mathrm{L}_{g(t)} \right)^*_e \left(\frac{\delta h_-}{\delta g}\right).
\end{equation}
\end{lemma}

By the same arguments used in the proof of Theorem \ref{thm-EP}, using the formulae of dual covariant derivative \eqref{eq:cov-curve-left-dual} and dual parallel transport \eqref{eq:parr-left-dual} instead, the equivalence between Hamilton's equations and the left Lie--Poisson equation follows directly.

\begin{theorem}[Left Lie--Poisson reduction and reconstruction]
Let \(G\) be equipped with a left-invariant connection whose associated left bilinear map is \(\nabla^{\mathrm{L}}\). The global Hamilton's equations 
\begin{equation}\label{Hamilton-Lie}
\left\{
\begin{aligned}
&g'(t) = \pt_p H,\\
&\frac{D}{dt}p(t) + p(t)\left( i_{g'(t)}T \right) = - d_g^\nabla H.
\end{aligned}
\right.
\end{equation}
are equivalent to the left Lie--Poisson equation \eqref{eq:LP-left}, via the reconstruction equations \(g'(t) = g(t) \frac{\delta h_-}{\delta \mu}\) and \(p(t) = d\left( \mathrm{L}_{g(t)^{-1}} \right)^*_{g(t)} \mu_-(t)\).
\end{theorem}

The reduced Legendre transform provides the link between the Lagrangian and Hamiltonian descriptions:
\[
\mu_- = \frac{\delta l_-}{\delta v},\qquad v_- = \frac{\delta h_-}{\delta \mu}, \qquad \mu_- v_- = h_-+l_-.
\]

Finally, by a completely analogous argument, there is a right-trivialized formalism, making use of right-invariant connections. We include the details in Appendix~\ref{sec-app-right-conn} for completeness.

\section{Cartan--Schouten connections}\label{sec-CS-conn}

In this section, we systematically study a distinguished family of invariant connections on Lie groups, known as Cartan--Schouten connections, for which the left bilinear map is proportional to the Lie bracket. As we shall see, this condition in fact implies bi-invariance, giving these connections a particularly left-right dual structure. The material developed here will also serve as a preparation for the connection-dependent variational principle introduced in the following section, in marked contrast to the previous section where the reduction equations were shown to be independent of the choice of connection.

\subsection{Cartan connections}

\begin{definition}[Cartan connection {\cite[Section 21.7]{GQ20}}]
A \emph{Cartan connection} on a Lie group $G$ is a left-invariant connection whose associated left bilinear map $\nabla^{\mathrm{L}}:\mathfrak g\times\mathfrak g\to\mathfrak g$ is skew-symmetric, i.e.
\(
\nabla^{\mathrm{L}}_v w + \nabla^{\mathrm{L}}_w v = 0, \forall v,w\in\mathfrak g.
\)
\end{definition}

Equivalently, the geodesics of $\nabla$ starting at the identity coincide with the one-parameter subgroups of $G$ \cite[Section 6.4]{Pos13}. To see this, we write down the geodesic equation of a left-invariant connection from \eqref{eq:cov-curve-left}:
\[
\left( \frac{d}{dt} + \nabla^{\mathrm{L}}_{g(t)^{-1} g'(t)} \right) g(t)^{-1} g'(t) = 0.
\]
If $g(\cdot)$ is a one-parameter subgroup, then $g(t)^{-1} g'(t) = g'(0)$ and the geodesic equation implies $\nabla^{\mathrm{L}}_{g'(0)} g'(0) = 0$. The arbitrariness of $g'(0)$ implies the skew-symmetry of $\nabla^{\mathrm{L}}$. Conversely, if $\nabla^{\mathrm{L}}$ is skew-symmetric, then the geodesic equation reduces to $\frac{d}{dt} (g(t)^{-1} g'(t)) = 0$, which together with the initial condition $g(0) = e$ is solved as the one-parameter subgroup $g(t) = \exp(t g'(0))$.

\subsection{Cartan--Schouten connections}

The most important family of Cartan connections is the Cartan--Schouten connections.

\begin{definition}[Cartan--Schouten connection]
Let \(\lambda\in\mathbb R\) be a fixed parameter. The \emph{Cartan--Schouten connection} (or \emph{\(\lambda\)-connection}) on a Lie group \(G\) is the unique left-invariant linear connection whose associated left bilinear map \(\nabla^{\mathrm{L}}:\mathfrak g\times\mathfrak g\to\mathfrak g\) is given by  
\begin{equation}\label{CS-conn}
  \nabla^{\mathrm{L}}_v w = \lambda \ad_v w = \lambda\,[v,w]_{\mathfrak g}.
\end{equation}
\end{definition}

Using formula \eqref{eq:left-conn} for left-invariant connections, we obtain the covariant derivative of arbitrary vector fields for Cartan--Schouten connections:
for vector fields \(V\) and \(W\) on \(G\),
\begin{equation*}
  (\nabla_V W)_g = g \left[ V_g(g^{-1} W_g) + \lambda \ad_{g^{-1} V_g} (g^{-1} W_g) \right].
\end{equation*}

\begin{lemma}[Covariant derivative along a curve]
Let \(g(\cdot)\) be a smooth curve on $G$. For a vector field \(W\) along \(g(\cdot)\),
\begin{equation}\label{cov-curve-Cartan-lambda}
  \begin{split}
    \frac{D}{dt} W_{g(t)} &= d\left( \mathrm{L}_{g(t)} \right)_e \left( \frac{d}{dt} + \lambda \ad_{d( \mathrm{L}_{g(t)^{-1}} )_{g(t)} g'(t)} \right) \left( d\left( \mathrm{L}_{g(t)^{-1}} \right)_{g(t)} W_{g(t)} \right) \\
    &= d\left( \mathrm{R}_{g(t)} \right)_e \left( \frac{d}{dt} + (\lambda-1) \ad_{d( \mathrm{R}_{g(t)^{-1}} )_{g(t)} g'(t)} \right) \left( d\left( \mathrm{R}_{g(t)^{-1}} \right)_{g(t)} W_{g(t)} \right).
  \end{split}
\end{equation}
For a 1-form \(\eta\) along \(g(\cdot)\),
\begin{equation}\label{cov-curve-Cartan-lambda-dual}
\begin{split}
  \frac{D}{dt} \eta_{g(t)} &= d\left( \mathrm{L}_{g(t)^{-1}} \right)^*_{g(t)} \left( \frac{d}{dt} - \lambda \ad_{d( \mathrm{L}_{g(t)^{-1}} )_{g(t)} g'(t)}^* \right) d\left( \mathrm{L}_{g(t)} \right)^*_e \eta_{g(t)} \\
  &= d\left( \mathrm{R}_{g(t)^{-1}} \right)^*_{g(t)} \left( \frac{d}{dt} - (\lambda-1) \ad_{d( \mathrm{R}_{g(t)^{-1}} )_{g(t)} g'(t)}^* \right) d\left( \mathrm{R}_{g(t)} \right)^*_e \eta_{g(t)}.
\end{split}
\end{equation}
\end{lemma}

\begin{proof}
The first equality of \eqref{cov-curve-Cartan-lambda} follows from \eqref{eq:cov-curve-left} and \eqref{CS-conn}. Next, since
\begin{equation*}
  d\left( \mathrm{L}_{g(t)^{-1}} \right)_{g(t)} W_{g(t)} = \Ad_{g(t)^{-1}} d\left( \mathrm{R}_{g(t)^{-1}} \right)_{g(t)} W_{g(t)},
\end{equation*}
we apply \eqref{der-Ad-curve-inv} and \eqref{Ad-bracket} to obtain
\begin{equation*}
  \begin{aligned}
    \frac{D}{dt} W_{g(t)} &= d\left( \mathrm{L}_{g(t)} \right)_e \bigg[ \frac{d}{dt} \left( \Ad_{g(t)^{-1}} d\left( \mathrm{R}_{g(t)^{-1}} \right)_{g(t)} W_{g(t)} \right) \\
    &\qquad\qquad\qquad + \lambda \ad_{g(t)^{-1} g'(t)} \left( \Ad_{g(t)^{-1}} d\left( \mathrm{R}_{g(t)^{-1}} \right)_{g(t)} W_{g(t)} \right) \bigg] \\
    &= d\left( \mathrm{L}_{g(t)} \right)_e \bigg[ \Ad_{g(t)^{-1}} \frac{d}{dt} \left( d\left( \mathrm{R}_{g(t)^{-1}} \right)_{g(t)} W_{g(t)} \right) - \Ad_{g(t)^{-1}} \ad_{g'(t) g(t)^{-1}} \left( d\left( \mathrm{R}_{g(t)^{-1}} \right)_{g(t)} W_{g(t)} \right) \\
    &\qquad\qquad\qquad + \lambda \Ad_{g(t)^{-1}} \ad_{g'(t) g(t)^{-1}} \left( d\left( \mathrm{R}_{g(t)^{-1}} \right)_{g(t)} W_{g(t)} \right) \bigg] \\
    &= d\left( \mathrm{R}_{g(t)} \right)_e \left[ \frac{d}{dt} \left( d\left( \mathrm{R}_{g(t)^{-1}} \right)_{g(t)} W_{g(t)} \right) + (\lambda-1) \ad_{g'(t) g(t)^{-1}} \left( d\left( \mathrm{R}_{g(t)^{-1}} \right)_{g(t)} W_{g(t)} \right) \right],
  \end{aligned}
\end{equation*}
which yields the second equality of \eqref{cov-curve-Cartan-lambda}.
Equation \eqref{cov-curve-Cartan-lambda-dual} follows from the argument of duality, similar to the proof of \eqref{eq:cov-curve-left-dual}.
\end{proof}


As a consequence of the second equality of \eqref{cov-curve-Cartan-lambda}, the covariant derivative in terms of the right action is: for vector fields \(V\) and \(W\)
\begin{equation}\label{eq:def-Cartan-lambda-right}
  (\nabla_V W)_g = \left[ V_g(W_g g^{-1}) + (\lambda-1) \ad_{V_g g^{-1}} (W_g g^{-1}) \right] g.
\end{equation}
We thus reach the following key observation for Cartan--Schouten connections.

\begin{proposition}[Bi-invariance of Cartan--Schouten connections]
The Cartan--Schouten connection of parameter \(\lambda\in\mathbb R\) is also right-invariant, and its associated right bilinear map is
\begin{equation}\label{CS-conn-right}
\nabla^{\mathrm{R}}_v w = (\lambda-1)\ad_v w = (\lambda-1)[v,w],\qquad v,w\in\mathfrak g.
\end{equation}
In particular, all Cartan--Schouten connections are bi-invariant.
\end{proposition}

\begin{proof}
Let \(V = v^{\mathrm R}, W = w^{\mathrm R}\) in \eqref{eq:def-Cartan-lambda-right}, with \(v,w\in\mathfrak g\). We have
\begin{equation*}
  (\nabla_{v^{\mathrm R}} w^{\mathrm R})_g = (\lambda-1) (\ad_v w) g.
\end{equation*}
Thus, the covariant derivative is right-invariant and induced right bilinear map is \eqref{CS-conn-right}.
\end{proof}

\begin{remark}\label{rem-left-right}
Comparing \eqref{CS-conn} and \eqref{CS-conn-right}, we see that the left and right formulations of a Cartan--Schouten connection differ merely by replacing the parameter \(\lambda\) with \(\lambda-1\). Recall from \eqref{Lie-bracket-left} and \eqref{Lie-bracket-right} the fact that the left-invariant fields yield the Lie bracket, while the right-invariant ones give the opposite bracket. Consequently, every formula derived for a Cartan--Schouten connection admits two versions, one in the left formalism and one in the right formalism, related by the substitutions
\[
v^{\mathrm L}\leftrightarrow v^{\mathrm R},\qquad 
\lambda \longleftrightarrow \lambda-1,\qquad 
[v,w]_{\mathfrak g}\longleftrightarrow -[v,w]_{\mathfrak g}.
\]
This duality is a hallmark of the Cartan--Schouten family.
\end{remark}

From Lemma \ref{cor-parr} and the previous remark, we immediately obtain the parallel transport for Cartan--Schouten connections.

\begin{corollary}[Parallel transport of the \(\lambda\)-connection]
Along a curve \(g(\cdot)\), the parallel transport operator \(\Gamma(g)_s^t:T_{g(s)}G\to T_{g(t)}G\) is given by
\begin{equation}\label{parr-tran-Cartan-lambda}
\begin{split}
  \Gamma(g)_s^t &= d\left( \mathrm{L}_{g(t)} \right)_e \mathcal T \exp\left( - \lambda \int_s^t \ad_{d( \mathrm{L}_{g(r)^{-1}} )_{g(r)} g'(r)} dr \right) d\left( \mathrm{L}_{g(s)^{-1}} \right)_{g(s)} \\
  &= d\left( \mathrm{R}_{g(t)} \right)_e \mathcal T \exp\left( (1-\lambda) \int_s^t \ad_{d( \mathrm{R}_{g(r)^{-1}} )_{g(r)} g'(r)} dr \right) d\left( \mathrm{R}_{g(s)^{-1}} \right)_{g(s)}.
\end{split}
\end{equation}
Its dual version \(\Gamma^*(g)_s^t:T^*_{g(s)}G\to T^*_{g(t)}G\) is given by
\begin{equation}\label{parr-tran-Cartan-lambda-dual}
\begin{split}
  \Gamma^*(g)_s^t &= d\left( \mathrm{L}_{g(t)^{-1}} \right)_{g(t)}^* \mathcal T \exp\left( \lambda \int_s^t \ad_{d( \mathrm{L}_{g(r)^{-1}} )_{g(r)} g'(r)}^* dr \right) d\left( \mathrm{L}_{g(s)} \right)_e^* \\
  &= d\left( \mathrm{R}_{g(t)^{-1}} \right)_{g(t)}^* \mathcal T \exp\left( (\lambda-1) \int_s^t \ad_{d( \mathrm{R}_{g(r)^{-1}} )_{g(r)} g'(r)}^* dr \right) d\left( \mathrm{R}_{g(s)} \right)_e^*.
\end{split}
\end{equation}
\end{corollary}

\begin{remark}\label{rem-simple}
When \(G\) is nilpotent, the parallel transport operator \eqref{parr-tran-Cartan-lambda} can be simplified. 
Let its Lie algebra \(\mathfrak g\) be nilpotent of step \(s\), i.e., the lower central series satisfies \(\mathfrak g^{s+1}=0\). Consequently, the series of the ordered exponential in \eqref{parr-tran-Cartan-lambda} and \eqref{parr-tran-Cartan-lambda-dual} truncate at the \(s\)-th term.
In the special case of a connected two-step nilpotent Lie group (such as the Heisenberg group), we have \(\ad_v\ad_w=0\) for all \(v,w\in\mathfrak g\), and thus, the parallel transport operator \eqref{parr-tran-Cartan-lambda} takes the explicit form
\begin{equation}\label{para-trans-heis}
\Gamma(g)_s^t
=
d\left( \mathrm{L}_{g(t)g(s)^{-1}} \right)_{g(s)}
- \lambda d\left( \mathrm{L}_{g(t)} \right)_e \int_s^t \ad_{d( \mathrm{L}_{g(r)^{-1}} )_{g(r)} g'(r)} dr\, d\left( \mathrm{L}_{g(s)^{-1}} \right)_{g(s)}.
\end{equation}
\end{remark}

The curvature and torsion of the \(\lambda\)-connection are easily computed from those of left-invariant connections Lemma \ref{cur-tor-left} and the formula of the left bilinear map \eqref{CS-conn}. By Remark \ref{rem-left-right}, there are right versions for them, using the corresponding right formulations.

For left-/right- invariant vector fields, the curvature and torsion are also left-/right- invariant. Thus, there is no ambiguity in restricting the curvature and torsion tensors for Cartan--Schouten connections to the Lie algebra:
\begin{equation}\label{tor-cur-lie-alg}
R(u,v)w = \lambda(\lambda-1)\,[[u,v]_{\mathfrak g},w]_{\mathfrak g},\qquad
T(v,w) = (2\lambda-1)[v,w]_{\mathfrak g},
\qquad u,v,w\in\mathfrak g.
\end{equation}

It can also be proved that, the curvature and torsion tensors for Cartan--Schouten connections are both covariantly constant, i.e., \(\nabla R=0, \nabla T=0\). Cf. \cite[Section 6.4]{Pos13}.

\begin{remark}
The Cartan--Schouten connections corresponding to the special parameter values \(\lambda=0,\frac12,1\) are of particular interest. \\
(i) The $(-)$-connection ($\lambda=0$) and $(+)$-connection ($\lambda=1$) are both flat. Indeed, for \(\lambda=0\), the parallel transport is simply left translation \(\Gamma(g)_s^t=d(\mathrm L_{g(t)g(s)^{-1}})_{g(s)}\), while for \(\lambda=1\), it is right translation \(\Gamma(g)_s^t=d(\mathrm R_{g(s)^{-1}g(t)})_{g(s)}\). In either case, the transport depends only on the endpoints, so the curvature vanishes identically, while the torsion is \(T=-[\cdot,\cdot]_{\mathfrak g}\) for \((+)\) and \(T=[\cdot,\cdot]_{\mathfrak g}\) for \((-)\). \\
(ii) The symmetric connection \(\lambda=\frac12\) is the unique torsion-free Cartan--Schouten connection. Its curvature is
\(
R(u,v)w=-\frac14\,[[u,v]_{\mathfrak g},w]_{\mathfrak g}.
\)
It is locally symmetric since \(\nabla R=0\). On a compact semisimple Lie group, when equipped with the bi-invariant metric induced by the Killing form, this connection coincides with the Levi-Civita connection and the group becomes a Riemannian symmetric space; in the non-compact case, the same holds after a sign change of the metric.
\end{remark}

\section{Connection-dependent variational principles}\label{sec-conn-vp}

In Section \ref{sec-geom-mech}, we considered Lagrangians depending on the instantaneous velocity \(\dot\gamma\). We now develop a new variational principle in which the Lagrangian depends on the velocity parallel-transported back to the initial point of the curve. This formulation reveals two distinct sources of nonlocality in the resulting Euler--Lagrange equations: a path-dependent term encoded in the parallel transport, and a future-dependent term arising from the integral of curvature.

\subsection{The transported velocity and its variation}

Let \(Q\) be a smooth manifold with a linear connection \(\nabla\) with torsion \(T\) and curvature \(R\). 

For a curve \(\gamma \in C^2([0,T], Q)\),
consider a one-parameter family \(\{\gamma_\epsilon\}\) of curves near it, with fixed endpoints \(\gamma_\epsilon(0)=\gamma(0)\), \(\gamma_\epsilon(T)=\gamma(T)\), and satisfying $\gamma_0 = \gamma$. The infinitesimal variation of $\gamma$ is as in \eqref{eq:inf-var}. We identify \(T_{\gamma(0)}Q\) with \(\mathbb R^m\) via a fixed linear isomorphism.

We define the \emph{transported velocity} of $\gamma$
\begin{equation}\label{trans-vel}
  v(t):=\Gamma(\gamma)_t^0\dot\gamma(t)\in T_{\gamma(0)}Q\simeq\mathbb R^m,
\end{equation}
where \(\Gamma(\gamma)_s^t:T_{\gamma(s)}Q\to T_{\gamma(t)}Q\) denotes parallel transport along \(\gamma\). Similarly,
\begin{equation}\label{trans-vel-family}
  v_\epsilon(t) := \Gamma(\gamma_\epsilon)_t^0\dot\gamma_\epsilon(t)\in T_{\gamma_\epsilon(0)}Q=T_{\gamma(0)}Q.
\end{equation}
The infinitesimal variation of $v$ is defined by
\begin{equation}\label{trans-vel-var}
  \delta v(t):=\frac{\pt}{\pt\epsilon}\bigg|_{\epsilon=0}v_\epsilon(t).
\end{equation}

We need an expression for \(\delta v\) in terms of \(\delta\gamma\). The following lemma is the key relation that replaces the simple formula \(\delta \dot \gamma=\frac{d}{dt} \delta\gamma\) in the flat case; its proof is given in Appendix~\ref{sec-app-selected-proofs}.

\begin{lemma}[Variation of transported velocity]\label{lem:variation-transported-velocity}
We have the following formula:
\begin{equation}\label{eq:delta-v-main}
\delta v(t)=\frac{d}{dt}\bigl(\Gamma(\gamma)_t^0\delta\gamma(t)\bigr)
-\Gamma(\gamma)_t^0 T\bigl(\gamma'(t),\delta\gamma(t)\bigr)
-\int_0^t \Gamma(\gamma)_r^0 \left( R\bigl(\gamma'(r),\delta\gamma(r)\bigr)\,\Gamma(\gamma)_t^r \gamma'(t)\right)dr.
\end{equation}
\end{lemma}

\subsection{Integro-differential Euler--Lagrange equation}

Let \(l:\mathbb \mathrm{R}_+\times Q\times \mathbb R^m\to\mathbb R\) be a smooth function. We consider the action functional
\begin{equation}\label{eq:action-conn-dep}
\mathcal S[\gamma]=\int_0^T l\bigl(t,\gamma(t),v(t)\bigr)\,dt
\end{equation}
defined for curves \(\gamma\in C^2([0,T],Q)\) satisfying fixed endpoint conditions \(\gamma(0)=q_0,\;\gamma(T)=q_T\), where $v(\cdot)$ is the transported velocity of the curve $\gamma$ defined in \eqref{trans-vel}.
Note that the Lagrangian $l$ is connection- and path-dependent.

\begin{proposition}[Integro-differential Euler--Lagrange equation]
A curve \(\gamma\in C^2([0,T],Q)\) is a critical point of the action \eqref{eq:action-conn-dep} under fixed-endpoint variations, where $v(\cdot)$ is the transported velocity \eqref{trans-vel}, if and only if \((\gamma(\cdot),v(\cdot))\) satisfies the integro-differential Euler--Lagrange equation
\begin{equation}\label{eq:int-diff-EL}
\begin{aligned}
\frac{\delta l}{\delta q}
= &\ \Gamma(\gamma)_0^t\left(\frac{d}{dt}\frac{\delta l}{\delta v}\right)
+ \frac{\delta l}{\delta v}\left(\Gamma(\gamma)_t^0 \left( T(\gamma'(t),\cdot)\right) \right) \\
&\ + \int_t^T \frac{\delta l}{\delta v}\bigl(r,\gamma(r),v(r)\bigr) \left(\Gamma(\gamma)_t^0 \left( R(\gamma'(t),\cdot)\,\Gamma(\gamma)_r^t\gamma'(r) \right) \right)dr.
\end{aligned}
\end{equation}
\end{proposition}

\begin{proof}
We compute the variation of the action \eqref{eq:action-conn-dep}. Using \eqref{eq:delta-v-main}, we have
\[
\begin{aligned}
\delta\mathcal S[\gamma] = \frac{d}{d\e}\bigg|_{\e=0} \mathcal S[\gamma_\e]
&=\int_0^T \left[\frac{\delta l}{\delta q}(\delta\gamma(t))+\frac{\delta l}{\delta v}(\delta v(t))\right]dt\\
&=\int_0^T \frac{\delta l}{\delta q}(\delta\gamma(t))\,dt
+\int_0^T \frac{\delta l}{\delta v}\left(\frac{d}{dt}(\Gamma(\gamma)_t^0\delta\gamma(t))\right)dt\\
&\quad -\int_0^T \frac{\delta l}{\delta v}\left(\Gamma(\gamma)_t^0 \left( T(\gamma'(t),\delta\gamma(t))\right) \right) dt\\
&\quad -\int_0^T \frac{\delta l}{\delta v}\left(\int_0^t \Gamma(\gamma)_r^0 \left(R(\gamma'(r),\delta\gamma(r))\,\Gamma(\gamma)_t^r\gamma'(t)\right)dr\right)dt.
\end{aligned}
\]
Integrating the second term at the right-hand side of the second equality by parts and using \(\delta\gamma(0)=\delta\gamma(T)=0\) as well as the duality of parallel transport, we obtain
\[
\int_0^T \frac{\delta l}{\delta v}\left(\frac{d}{dt}(\Gamma(\gamma)_t^0\delta\gamma(t))\right)dt
= -\int_0^T \left(\frac{d}{dt}\frac{\delta l}{\delta v}\right)\bigl(\Gamma(\gamma)_t^0\delta\gamma(t)\bigr)\,dt = -\int_0^T \Gamma(\gamma)_t^0\left(\frac{d}{dt}\frac{\delta l}{\delta v}\right)\bigl(\delta\gamma(t)\bigr)\,dt.
\]
For the curvature term, we swap the order of integration:
\[
\begin{aligned}
&\ \int_0^T \frac{\delta l}{\delta v}\bigl(t,\gamma(t),v(t)\bigr) \left(\int_0^t \Gamma(\gamma)_r^0 \left( R(\gamma'(r),\delta\gamma(r))\,\Gamma(\gamma)_t^r\gamma'(t)\right)dr\right)dt \\
=&\ \int_0^T \int_t^T \frac{\delta l}{\delta v}\bigl(r,\gamma(r),v(r)\bigr) \left( \Gamma(\gamma)_t^0 \left( R(\gamma'(t),\delta\gamma(t))\,\Gamma(\gamma)_r^t\gamma'(r)\right)\right)dr\,dt.
\end{aligned}
\]
Collecting all terms, we get
\[
\begin{aligned}
\delta\mathcal S[\gamma]
=\int_0^T \Bigg[ &
\frac{\delta l}{\delta q}
-\Gamma(\gamma)_0^t\left(\frac{d}{dt}\frac{\delta l}{\delta v}\right)
-\frac{\delta l}{\delta v}\left(\Gamma(\gamma)_t^0 \left(T(\gamma'(t),\cdot)\right)\right)\\
& -\int_t^T \frac{\delta l}{\delta v}\bigl(r,\gamma(r),v(r)\bigr) \left(\Gamma(\gamma)_t^0 \left(R(\gamma'(t),\cdot)\,\Gamma(\gamma)_r^t\gamma'(r)\right)\right)dr
\Bigg](\delta\gamma(t))\,dt.
\end{aligned}
\]
Since \(\delta\gamma(t)\) is arbitrary, we obtain the desired result.
\end{proof}

\begin{remark}
(i). Equation \eqref{eq:int-diff-EL} is an integro-differential equation, involving an integral over the remaining part of the curve. This term arises from the curvature and represents a future-dependence effect: the dynamics at time \(t\) depend on the future evolution of the transported velocity. 
The presence of the ``future'' integral is not a violation of causality, but rather a manifestation of the fact that the Euler--Lagrange equation for a nonlocal action naturally yields a two-point boundary value problem. When we change variables from the standard \(\dot\gamma(t)\) to \(v(t)=\Gamma(\gamma)_t^0\dot\gamma(t)\), the nonlocality of the variable is transferred to the equation.

(iii). A local-coordinate formulation of the integro-differential Euler--Lagrange equation \eqref{eq:int-diff-EL} using a parallel frame can be found in Appendix \ref{sec-app-eqn}.
\end{remark}


\subsection{Path-dependent reduction on Lie groups}

We now apply the connection-dependent variational principle to the case of a Lie group \(G\) equipped with a Cartan--Schouten connection of parameter \(\lambda\in[0,1]\). Let \(\mathfrak g\) be the Lie algebra of \(G\).

For a curve \(g(\cdot)\) on $G$, define its transported velocity
\[
v(t):= \Gamma(g)_t^0\dot g(t)\in T_{g(0)}G \simeq\mathfrak g.
\]
Using the explicit expression for the parallel transport in a Cartan--Schouten connection, this can be written as
\[
v(t)=g(0) \mathcal T\exp\left(\lambda\int_0^t \ad_{g(r)^{-1}g'(r)}dr\right)g(t)^{-1}\dot g(t),
\]
where \(\mathcal T\exp\) denotes the ordered exponential. We also defined the transported infinitesimal variation by
\[
w(t):= \Gamma(g)_t^0\delta g(t)
=g(0) \mathcal T\exp\left(\lambda\int_0^t \ad_{g(r)^{-1}g'(r)}dr\right)g(t)^{-1}\delta g(t),
\]
with \(w(0)=w(T)=0\).

\begin{lemma}
We have the following formula:
\begin{equation}\label{eq:delta-v-lie-lambda}
\delta v(t) = \dot w(t) - (2\lambda-1) \ad_{v(t)} w(t) - \lambda (1-\lambda) \ad_{v(t)} \int_0^t \ad_{v(r)} w(r) dr.
\end{equation}
\end{lemma}

\begin{proof}
Applying the general formula \eqref{eq:delta-v-main} and the fact that $\nabla R=0$ and $\nabla T=0$,
\begin{equation*}
  \begin{aligned}
    \delta v(t) &= \dot w(t) - \Gamma(\gamma)_t^0 \big( T(\dot g(t),\delta g(t)) \big) - \int_0^t \Gamma(g)^0_r \big( R\left( \dot g(r), \delta g(r) \right) \Gamma(g)_t^r \dot g(t) \big) dr \\
    &= \dot w(t) - T(v(t), w(t)) - \int_0^t R\left( v(r), w(r) \right) v(t) dr.
  \end{aligned}
\end{equation*}
Then, using the explicit expressions in \eqref{tor-cur-lie-alg} for the torsion and curvature of the Cartan--Schouten connection,
we obtain the desired formula.
\end{proof}


\begin{theorem}[Integro-differential Euler--Poincar\'e equation]
A curve $g(\cdot)\in C^2([0,T],G)$ is a critical point of the action \(\mathcal S[g]=\int_0^T l(t,g(t),v(t))dt\) under fixed-endpoint variations if and only if \((g(\cdot),v(\cdot))\) satisfies the integro-differential Euler--Poincar\'e equation
\begin{equation}\label{eq:reduced-int-diff-EP}
\Gamma^*(g)_t^0\frac{\delta l}{\delta g}
=\frac{d}{dt}\frac{\delta l}{\delta v}
+(2\lambda-1)\ad_{v(t)}^*\frac{\delta l}{\delta v}
+\lambda(1-\lambda)\ad_{v(t)}^*\int_t^T \ad_{v(r)}^*\frac{\delta l}{\delta v}\bigl(r,g(r),v(r)\bigr) \,dr.
\end{equation}
\end{theorem}

\begin{proof}
Substituting this into the general Euler--Lagrange equation \eqref{eq:int-diff-EL} gives
\[
\begin{aligned}
\frac{\delta l}{\delta g}
&=\Gamma(g)_0^t \left(\frac{d}{dt}\frac{\delta l}{\delta v}\right)
+\frac{\delta l}{\delta v}\left(\Gamma(g)_t^0 \big( T(\dot g(t),\cdot) \big) \right)\\
&\quad +\int_t^T \frac{\delta l}{\delta v}\bigl(r,g(r),v(r)\bigr) \left(\Gamma(g)_t^0 \big( R(\dot g(t),\cdot)\,\Gamma(g)_r^t\dot g(r) \big) \right)dr.
\end{aligned}
\]
Now, for any \(u\in T_{g(t)}G\), using the invariance of the curvature and torsion under left translation (and the fact that \(\nabla R=0\), \(\nabla T=0\)), we obtain
\[
\Gamma^*(g)_t^0\frac{\delta l}{\delta g}
=\frac{d}{dt}\frac{\delta l}{\delta v}
+\frac{\delta l}{\delta v}\big(T(v(t),\cdot)\big) +\int_t^T \frac{\delta l}{\delta v}\bigl(r,g(r),v(r)\bigr) \big(R(v(t),\cdot)v(r)\big)dr.
\]
Finally, using the explicit expressions \eqref{tor-cur-lie-alg}, we obtain the reduced equation.
\end{proof}

\begin{remark}
When \(\lambda=0\), equation \eqref{eq:delta-v-lie-lambda} coincides with \eqref{eq:delta-v-left} and equation \eqref{eq:reduced-int-diff-EP} reduces to the standard left Euler--Poincar\'e equation \eqref{eq:EP-left}.
When \(\lambda=1\), 
\eqref{eq:reduced-int-diff-EP} reduces to the right Euler--Poincar\'e equation.
For the torsion-free case \(\lambda=1/2\), the first-order term vanishes, but the curvature integral term remains, providing a path-dependent correction that is absent in the classical flat cases.
\end{remark}

\subsubsection*{Reformulation as a standard two-point boundary value problem}

The integro-differential equation can be converted into a system of ordinary differential equations by introducing an auxiliary variable that absorbs the future integral. Define
\[
\xi(t):=\int_t^T \ad_{v(r)}^*\frac{\delta l}{\delta v}(r)\,dr \in \mathfrak g^*.
\]
Then \(\xi(t)\) satisfies a terminal value problem
Substituting this into \eqref{eq:reduced-int-diff-EP}, together with the reconstruction equation \(\dot g(t)=\Gamma(g)_0^t v(t)\) and the reduced Legendre transform $\mu = \frac{\delta l}{\delta v}$, yields the following system:

\begin{equation}\label{eq:reduced-system}
\left\{
\begin{aligned}
&\Gamma^*(g)_t^0\frac{\delta l}{\delta g} = \dot\mu(t)
+(2\lambda-1)\ad_{v(t)}^*\mu(t)
+\lambda(1-\lambda)\ad_{v(t)}^*\xi(t), \\
&\dot\xi(t)=-\ad_{v(t)}^*\mu(t),\qquad \xi(T)=0, \\
&\dot g(t)=\Gamma(g)_0^t v(t),\qquad g(0)=g_0,g(T)=g_T, \\
&\mu(t) = \frac{\delta l}{\delta v}(t).
\end{aligned}
\right.
\end{equation}
This is now a system of ordinary differential equations coupled with mixed conditions. The unknowns are \(g(\cdot)\), \(v(\cdot)\), and \(\mu(\cdot)\), with boundary conditions prescribed at both ends of the interval.

\section{Solvability of the reduced integro-differential equation}\label{sec-solv}

The reduced integro-differential Euler--Poincar\'e equation \eqref{eq:reduced-int-diff-EP} derived in the previous section has an unusual structure: the left-hand side depends on the history of the curve (through the parallel transport from \(0\) to \(t\)), while the right-hand side contains an integral over the future interval \([t,T]\). This may raise concerns about the well-posedness and solvability of the equation. In this section we show that, despite its nonlocal appearance, the equation can be solved by standard methods. The nilpotent example of the three-dimensional Heisenberg group is discussed in Appendix \ref{sec-app-H3}.

\subsection{Discretizations}

We first discretize the reconstruction equation on the Lie group and then formulate a two-level scheme for the resulting coupled forward--backward system.

\subsubsection*{Discretization of the reconstruction equation}

The reconstruction equation
\(
\dot g(t)=\Gamma(g)_0^t v(t)
\)
can be interpreted as a differential equation on the Lie group \(G\) of the form
\(
\dot g(t)=g(t)\,u(t),
\)
where
\begin{equation}\label{body-v}
  u(t):= \mathcal T \exp\left( - \lambda \int_0^t \ad_{g(r)^{-1} g'(r)} dr \right) g(0)^{-1} v(t)\in\mathfrak g
\end{equation}
is the body velocity at \(g(t)\). This is a time-dependent linear ODE on the Lie algebra. A standard geometric integration approach is the Lie--Euler method, to approximate the curve \(g(\cdot)\) by a piecewise geodesic (or piecewise one-parameter subgroup) on each small time interval \([t_k,t_{k+1}]\). Specifically, on each interval we replace the body velocity \(u(t)\) by a constant value \(u_k\), so that the solution over that interval is approximated by the group exponential: $g(t_k) \approx g_k$ and
\begin{equation}\label{disc-grp}
g_{k+1} = g_k\exp(h\,u_k),\qquad g_0=g(0), g_N=g(T),
\end{equation}
for \(k=0,\dots,N-1\), with step size \(h=t_{k+1}-t_k\). 

To implement this approximation, we need to define \(u_k\). From definition \eqref{body-v}, we have
\begin{equation}\label{disc-body-v}
u_k = P_k g_0^{-1} v_k,
\end{equation}
where $v_k = v(t_k)$ is the discretization of the known curve $v(\cdot)$, \(P_k: \mathfrak g\to \mathfrak g\) is the discrete approximation of the ordered exponential
\[
P_k \approx \mathcal T\exp\left(-\lambda\int_0^{t_k}\ad_{g(r)^{-1}g'(r)}\,dr\right).
\]
The update of \(P_k\) follows from the same piecewise-geodesic approximation: on the interval \([t_k,t_{k+1}]\), the body velocity is constant equal to \(u_k\), so the ordered exponential becomes an ordinary exponential:
\begin{equation}\label{disc-exp}
  P_{k+1}=\exp\left(-\lambda h\,\ad_{u_k}\right) P_k, \qquad P_0 = I.
\end{equation}

\begin{remark}
This piecewise-geodesic approximation respects the fact that, for a Cartan--Schouten connection, the geodesics are exactly the one-parameter subgroups. 
\end{remark}

\subsubsection*{Discretization of the momentum and auxiliary equations}

We discretize the Legendre transform
\begin{equation}\label{disc-Legendre}
\mu_k:=\frac{\delta l}{\delta v}(t_k,g_k,v_k)\in\mathfrak g^*
\end{equation}
The backward equation \(\dot\xi=-\ad_v^*\mu\) with \(\xi(T)=0\) is discretized by a backward Euler step:
\begin{equation}\label{disc-aux}
\xi_k=\xi_{k+1}+h\,\ad_{v_k}^*\mu_k,\qquad \xi_N=0.
\end{equation}
Thus, once the sequences \(\{v_k\}\) and \(\{\mu_k\}\) are known (or guessed), all \(\xi_k\) are determined by backward recursion.

The momentum equation (the first of \eqref{eq:reduced-system}) is discretized using an implicit Euler step for the derivative of \(\mu\). The left-hand side requires the parallel transport of the covector \(\frac{\delta l}{\delta g}\) from \(t\) back to \(0\). In the discrete setting, this is achieved via the dual of the operator $P_k$. Define
\[
\eta(t):= \Gamma^*(g)_t^0 \frac{\delta l}{\delta g},
\]
which is a covector at \(g_0\). In terms of the discrete data, we have that $\eta(t_k)$ is approximated by
\begin{equation}\label{disc-mmt}
\eta_k = d\big(\mathrm{L}_{g_0}\big)^* P_k^* d\big(\mathrm{L}_{g_k^{-1}}\big)^* \left(\frac{\delta l}{\delta g}(t_k,g_k,v_k)\right),
\end{equation}
where \(P_k^*\) is the dual of the linear map \(P_k\). 
The implicit Euler discretization of the momentum equation then reads,
\begin{equation}\label{disc-mmt-eqn}
\eta_k = \frac{\mu_{k+1}-\mu_k}{h} + (2\lambda-1)\ad_{v_k}^*\mu_k + \lambda(1-\lambda)\ad_{v_k}^*\xi_k.
\end{equation}
Equations \eqref{disc-grp}--\eqref{disc-mmt-eqn} form a closed algebraic system for the unknowns \(\{g_k,v_k,\mu_k,\xi_k,\eta_k\}\).

\subsection{A two-level scheme}\label{subsec-num-sch}

We present a general numerical scheme for solving the reduced integro-differential Euler--Poincar\'e equation \eqref{eq:reduced-system}. The scheme is a combination of a shooting method for the initial velocity \(v(0)\), and a fixed‑point iteration for the auxiliary variables \(\xi\).
It applies to any Lie group, including non-nilpotent ones such as \(\mathrm{SO}(3)\).

The boundary conditions are \(g_0\) prescribed and \(g_N\) required to equal a given \(g_T\). We solve the system using a combination of a shooting method for the initial velocity \(v_0\in\mathfrak g\), and a fixed‑point iteration for the auxiliary variables \(\{\xi_k\}\).

Step 1. Initialization: Choose an initial guess for \(v_0\) (e.g., \(v_0=0\), or the body velocity that would connect \(g_0\) to \(g_T\) along a geodesic). Initialize the auxiliary sequence \(\xi_k^{\text{init}}=0\) for all \(k=0,\dots,N-1\) (corresponding to neglecting the curvature correction term).

Step 2. Forward integration (inner loop): Given the current guess for \(\{\xi_k\}\), we determine the discrete trajectory \(\{g_k,v_k,\mu_k\}\) sequentially as follows. 
\begin{itemize}
  \item Assume we are at step \(k\) with known \(g_k\), \(v_k\), \(\xi_k\), and $P_k$. 
  \item Compute the body velocity \(u_k\) via \eqref{disc-body-v}, $\mu_k$ via \eqref{disc-Legendre}, and \(\eta_k\) via \eqref{disc-mmt}.
  \item Update the group element $g_{k+1}$ via \eqref{disc-grp} and the ordered exponential $P_{k+1}$ via \eqref{disc-exp}.
  \item Solve the implicit Euler equation \eqref{disc-mmt-eqn} for \(\mu_{k+1}\).
  \item Invert the Legendre transform \eqref{disc-Legendre} (as the Lagrangian is regular) to obtain \(v_{k+1}\) from \(g_{k+1}, \mu_{k+1}\).
\end{itemize}
   This completes one forward step. Thus, the entire sequence \(\{g_k,v_k,\mu_k\}\) is generated recursively.

Step 3. Shooting adjustment (inner loop): After the forward integration, we obtain the final group element \(g_N\) as a function of the initial guess \(v_0\). We compare \(g_N\) with the prescribed \(g_T\) and update \(v_0\) using a Newton method. The Jacobian \(\partial g_N/\partial v_0\) is approximated by finite differences. This inner iteration continues until the terminal error \(\|g_N-g_T\|\) is below a tolerance. The result is a trajectory \(\{g_k,v_k,\mu_k\}\) that satisfies the terminal condition $g(T) =g_T$ for the current \(\{\xi_k\}\).

Step 4. Update of auxiliary variables (outer loop): With the new \(\{v_k,\mu_k\}\) from the inner loop, we recompute the auxiliary sequence \(\{\xi_k\}\) using the backward recurrence \eqref{disc-aux}, starting from \(\xi_N=0\) and going backwards to \(k=0\). This yields an updated sequence \(\{\xi_k^{\text{new}}\}\).

Step 5. Outer fixed‑point iteration: Replace the old \(\{\xi_k\}\) by the new values and repeat the entire procedure (Steps 2--4) until the changes in \(\xi_k\) (and hence in the trajectory) fall below a tolerance. In practice, this outer iteration converges quickly for small step sizes \(h\).

Step 6. Convergence and output: Once the outer iteration has converged, the sequences \(\{g_k,v_k\}\) form a discrete solution of the integro‑differential system that satisfies both the momentum equation and the terminal condition \(g_N=g_T\). 

This two‑level structure, inner shooting for \(v_0\) and outer fixed‑point iteration for \(\{\xi_k\}\), decouples the local dynamics from the non‑local terms, making the algorithm robust and easy to implement. The pseudocode of the two‑level scheme and more discussions about its accuracy and improvement can be found in Appendix \ref{sec-app-scheme}.

\subsection{Application to the rigid body}\label{subsec-rigid-body}

We now specialize the general numerical scheme of the preceding subsection to the non-nilpotent rotation group \(\mathrm{SO}(3)\), the configuration space for the motion of a rigid body.
We use the standard identifications \(\mathfrak{so}(3)\cong\mathbb R^3\) and \(\mathfrak{so}(3)^*\cong\mathbb R^3\) via the hat and breve maps. The basis conventions are collected in Appendix~\ref{sec-app-rigid-body}.

\subsubsection*{The free rigid body}

We first consider the free rigid body, i.e. a Lagrangian consisting only of kinetic energy:
\[
l(g,v) = \frac12\, v\cdot \mathbb I v,
\]
where \(\mathbb I = \operatorname{diag}(I_1,I_2,I_3)\) is the inertia tensor. The momentum is \(\mu = \mathbb I v\). Since the Lagrangian does not depend on \(g\), the reduced Euler--Poincar\'e equation \eqref{eq:reduced-int-diff-EP} includes the curvature correction term:
\[
0 = \mathbb I \dot v - (2\lambda-1) v\times (\mathbb I v) - \lambda(1-\lambda)\, v\times \xi,
\]
with \(\dot\xi = - v\times (\mathbb I v)\), \(\xi(T)=0\). The reconstruction of the attitude \(g(\cdot)\) from the body velocity \(v(\cdot)\) is given by \(\dot g = \Gamma(g)_0^t v\), \(g(0)=g_0\), \(g(T)=g_T\).
\begin{remark}
For the $(-)$-connection (\(\lambda=0\)),
the equation becomes
\(
\mathbb I \dot v + v\times (\mathbb I v) = 0,
\)
which is the classical Euler equation for a free rigid body \cite[Section 1.5]{HSS09}.
\end{remark}

For brevity, the harmonic oscillator on \(\mathrm{SO}(3)\) and the complete discretization are given in Appendix~\ref{sec-app-rigid-body}. We now turn directly to the numerical comparison of the two models.

\subsubsection*{Numerical examples}

We now present numerical results for the free rigid body and the harmonic oscillator Lagrangian on \(\mathrm{SO}(3)\). The asymmetric inertia tensor we use is \(\mathbb I=\mathrm{diag}(2,3,4)\). 
The discrete systems detailed in Appendix~\ref{sec-app-rigid-body} are solved using the two-level algorithm described in Subsection \ref{subsec-num-sch}, with the inner shooting loop implemented via the Levenberg--Marquardt method.
The start and end points on the unit sphere are chosen and the target attitude \(g_T\in\mathrm{SO}(3)\) is the unique rotation mapping \(a_0\) to \(a_T\).
The start and end points are marked with black dots. Red dashed: \(\lambda=0\); green solid: \(\lambda=\frac12\); blue dotted: \(\lambda=1\).

\paragraph{Free rigid body.}
For the free rigid body, the Lagrangian is purely kinetic: \(l(g,v)=\frac12 v\cdot \mathbb I v\). Figure \ref{fig:free_trajectories} shows the space trajectories of the body-fixed vector \(a(t)=g(t)a_0\) on the unit sphere. 
The start and end points on the unit sphere are chosen as
\[
a_0=\frac{(-1,-1,-1)}{\sqrt{3}},\qquad a_T=\frac{(0.9,1.0,1.0)}{\sqrt{2.81}},
\]
The integration time is \(T=50\) with \(N=1000\) steps, and the tolerance is set to \(10^{-8}\).
For \(\lambda=0\) and \(\lambda=1\), the trajectories are symmetric, reflecting the left-right symmetry of the free Euler equations. The \(\lambda=\frac12\) trajectory exhibits a small drift due to the curvature correction term, but remains close to the symmetric ones.

\paragraph{Harmonic oscillator.}
For the harmonic oscillator, the potential \(V(g)=\frac12(3-\operatorname{tr}g)\) is added. Figure~\ref{fig:harmonic_trajectories} displays the corresponding trajectories. 
The start and end points on the unit sphere are chosen as
\[
a_0=\frac{(-0.7, -1.0, -0.7)}{\sqrt{1.98}},\qquad a_T=\frac{(0.8, 0.8, 1.0)}{\sqrt{2.28}},
\]
The integration time is \(T=12\) with \(N=800\) steps, and the tolerance is set to \(10^{-6}\).
The symmetry between \(\lambda=0\) and \(\lambda=1\) observed in the free case is now broken: the potential gradient term \(\eta_k\) does not change sign under \(\lambda\mapsto 1-\lambda\).

\begin{figure}[htbp]
\centering
\begin{subfigure}[b]{0.49\textwidth}
\includegraphics[width=\textwidth]{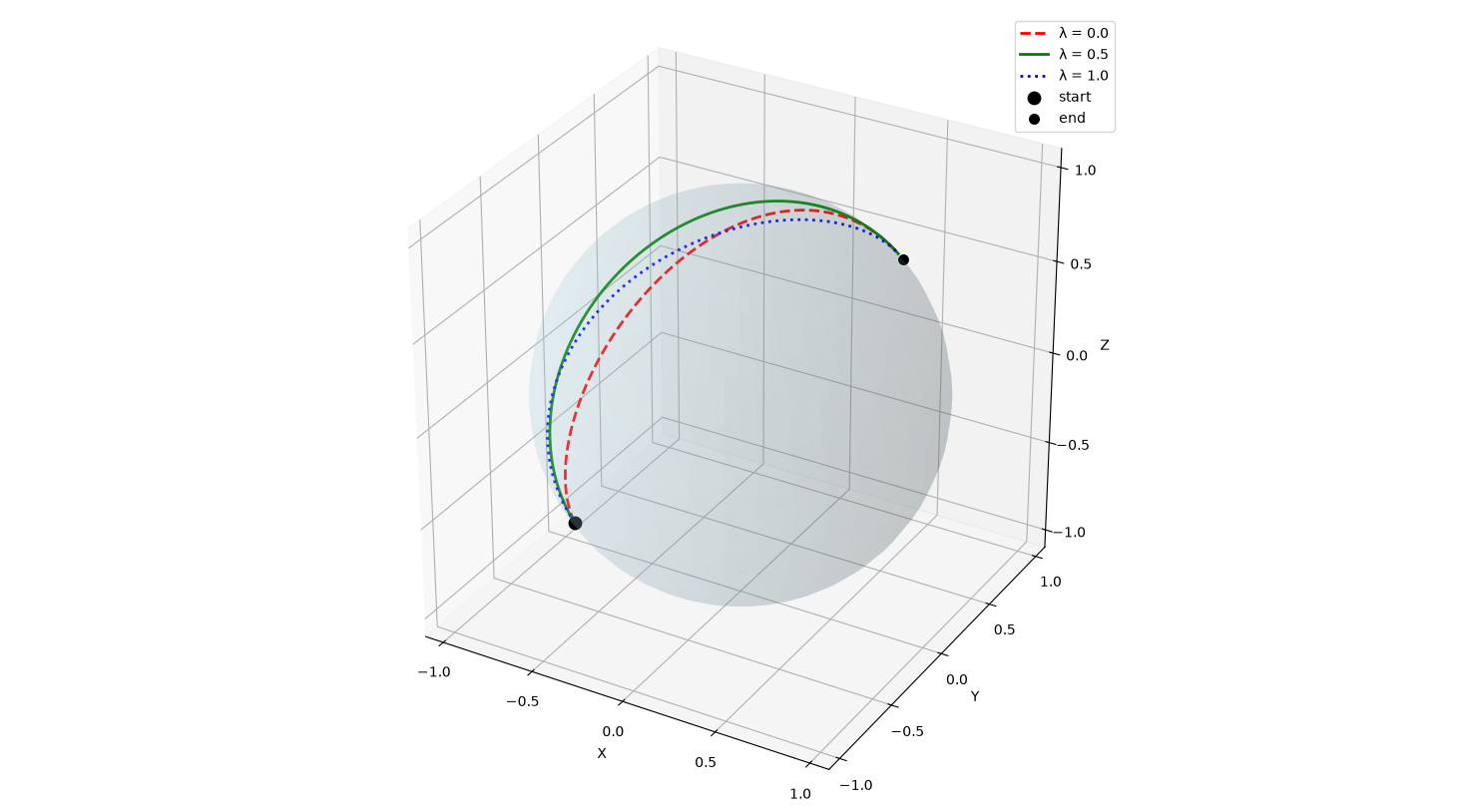}
\caption{}
\label{fig:free_trajectories}
\end{subfigure}
\hfill
\begin{subfigure}[b]{0.49\textwidth}
\includegraphics[width=\textwidth]{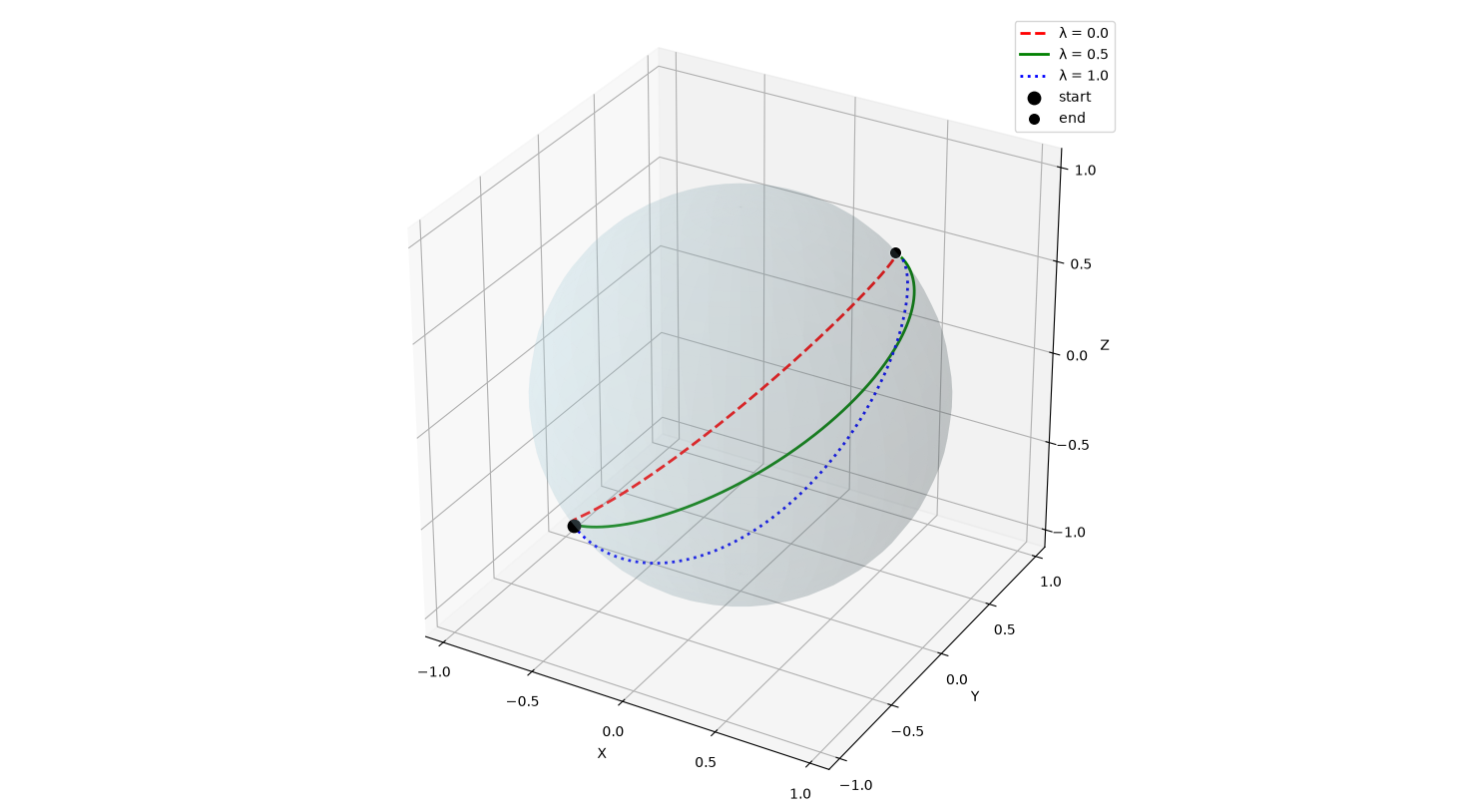}
\caption{}
\label{fig:harmonic_trajectories}
\end{subfigure}
\caption{Trajectories on the unit sphere for the free rigid body (left) and the harmonic oscillator (right).}
\label{fig:trajectories}
\end{figure}



\section{Conclusion and outlook}\label{sec-conclusion}

We have studied the role of invariant affine connections in geometric mechanics on Lie groups from two complementary viewpoints. For ordinary Lagrangian and Hamiltonian systems, we formulated the global Euler--Lagrange and Hamilton equations using an arbitrary left- or right-invariant connection and established their equivalence with the Euler--Poincar\'e and Lie--Poisson equations. Although the global formulations explicitly involve the connection, all connection-dependent terms cancel after reduction. This explains geometrically why the classical reduced equations are independent of the chosen invariant connection. For Cartan--Schouten connections, we then introduced a connection-dependent variational principle based on velocities parallel-transported to the initial point. The resulting integro-differential Euler--Poincar\'e equation exhibits both path dependence through parallel transport and future dependence through curvature. By introducing an auxiliary variable, the curvature integral can be converted into a terminal-value equation, revealing the underlying two-point boundary value structure. The two-level scheme and the examples on the Heisenberg group and \(\mathrm{SO}(3)\) demonstrate that this nonlocal system can be treated numerically while retaining its Lie-group geometry.
To study the stochastic case of geometric reduction, connections are indispensable: even the basic notion of a manifold-valued martingale, the stochastic counterpart of a conserved quantity, is defined relative to a connection, while stochastic variational principles likewise rely essentially on covariant constructions \cite{HZ23,HZ23b}. Thus, our analysis of invariant connections on Lie groups provides a natural foundation for stochastic reduction.

\section*{Acknowledgments}
The work of Q.~Huang is supported by the National Natural Science Foundation of China under Grant No. 12501241, the Basic Research Program of Jiangsu under Grant No. BK20251280, the Fundamental Research Funds for the Central Universities under Grant No. 2242026RCB0008, the Jiangsu Provincial Scientific Research Center of Applied Mathematics under Grant No. BK20233002 and the Start-up Research Fund of Southeast University under Grant No. RF1028624194.

\begin{appendices}

Appendix \ref{sec-app-selected-proofs} contains the proofs of selected results deferred from the main text. The right versions of invariant connections and reduction results are summarized in Appendix \ref{sec-app-right-conn}, Appendix \ref{sec-app-eqn} is devoted to a local-coordinate formulation of the integro-differential Euler--Lagrange equation, Appendix \ref{sec-app-H3} explains the nilpotent example of the three-dimensional Heisenberg group, Appendix \ref{sec-app-rigid-body} collects the Lie-algebra conventions, harmonic-oscillator model, and discrete equations used in the rigid-body computations, and Appendix \ref{sec-app-scheme} provides more details of the numerical algorithm.

\section{Proofs of selected results}\label{sec-app-selected-proofs}

\begin{proof}[Proof of Lemma~\ref{lem:horizontal-differential}]
For any tangent vector \(v\in T_xM\) and any point \(e\in E_x\), by \eqref{eq:hor-diff-local} and \eqref{eq:hrz-lift-vec},
\begin{equation}\label{hor-diff-vf}
d_x^\Gamma F|_e (v) = \Gamma^*(dF|_e)(v) = dF|_e(\Gamma(v)) = dF|_e(\bar v).
\end{equation}
Let \(\bar\gamma\) be a horizontal lift of $\gamma$ through \(e\in E_{\gamma(t)}\) at time $t$. Since \(\overline{\gamma'} = \bar\gamma'\) by \eqref{eq:hrz-lift-curve}, using \eqref{hor-diff-vf},
\[
d_x^\Gamma F|_e (\gamma'(t)) = dF|_e\bigl(\overline{\gamma'(t)}\bigr) = dF|_e(\bar\gamma'(t)) = \frac{d}{dt} F(\bar\gamma(t)) = \frac{d}{d\e}\bigg|_{\e=0} F\bigl(\Gamma(\gamma)_t^{t+\e} \bar\gamma(t)\bigr) = \frac{d}{d\e}\bigg|_{\e=0} F\bigl(\Gamma(\gamma)_t^{t+\e} e\bigr).
\]
The penultimate equality follows because \(\bar\gamma(t+\e) = \Gamma(\gamma)_t^{t+\e} \bar\gamma(t)\) by definition of horizontal lifts.
\end{proof}

\begin{proof}[Proof of Proposition~\ref{prop:hamilton-principle-phase}]
Let \(q_\epsilon(\cdot)\) be a variation of \(q(\cdot)\) with \(\delta q(t)=\frac{\partial}{\partial\epsilon}\big|_{\epsilon=0}q_\epsilon(t)\) and \(\delta q(0)=\delta q(T)=0\). Write \(p_\epsilon(\cdot)\) for the corresponding variation of the 1-form. Then by the chain rule \eqref{chain-rule-2},
\[
\begin{aligned}
\delta\overline{\mathcal S} [q(\cdot), p(\cdot)] &= \int_0^T \frac{\pt}{\pt\e}\bigg|_{\e=0} \bigl[p_\e(t)(\dot q_\e(t))-H(q_\e(t),p_\e(t),t)\bigr]dt \\
&= \int_0^T\left[\nabla_{\delta q}p(\dot q)+p(\delta\dot q)-d_q^\nabla H(\delta q)-\nabla_{\delta q}p(\pt_p H)\right] dt
\end{aligned}
\]
Using \eqref{eq:variation-der} for \(\delta\dot q\) and integrating by parts the term containing \(p\big(\frac{D}{dt} \delta q\big)\) gives
\[
\delta\overline{\mathcal S}=\int_0^T\Bigl[\nabla_{\delta q}p\bigl(\dot q-\pt_p H\bigr)
-\Bigl(\frac{D}{dt}p+p\left( i_{\dot q}T \right)+d_q^\nabla H\Bigr)(\delta q)\Bigr]dt .
\]
Since \(\delta q\) and the variation of \(p\) are independent, both parentheses must vanish, yielding \eqref{eq:HE-1}.
\end{proof}

\begin{proof}[Proof of Lemma~\ref{lem:left-lie-poisson}]
In reduced variables, the phase-space action becomes
\[
\overline{\mathcal S}= \int_0^T \bigl[\mu_-(v_-) - h_-(g,\mu_-)\bigr]dt,
\]
where \(v_-=g^{-1}\dot g\) as before. Take a variation \((\delta g,\delta\mu_-)\) with \(\delta g(0)=\delta g(T)=0\) and define \(w_-=g^{-1}\delta g\) as in \eqref{eq:w-left}. Using \eqref{eq:delta-v-left},
\[
\begin{split}
  \delta\overline{\mathcal S} &= \int_0^T \left[ \delta\mu_-(v_-) + \mu_-(\delta v_-) - \frac{\delta h_-}{\delta g}(\delta g) - \delta\mu_-\left( \frac{\delta h_-}{\delta \mu} \right) \right]dt \\
  &= \int_0^T \left[ \delta\mu_-(v_-) + \mu_-(\dot w_- + \ad_{v_-} w_-) - \frac{\delta h_-}{\delta g}(g w_-) - \delta\mu_-\left( \frac{\delta h_-}{\delta \mu} \right) \right]dt.
\end{split}
\]
Integrate by parts the term \(\mu_-(\dot w_-)\) and using \(w_-(0)=w_-(T)=0\):
\[
\int_0^T \mu_-(\dot w_-)\,dt = -\int_0^T \dot\mu_-(w_-)\,dt .
\]
Rewrite \(\frac{\delta h_-}{\delta g}(g w_-)= d\big( \mathrm{L}_{g(t)} \big)^*_e \big(\frac{\delta h_-}{\delta g}\big)(w_-)\) and \(\mu_-(\ad_{v_-} w_-) = \ad_{v_-}^*\mu_- (w_-)\). Collecting the coefficients of \(\delta\mu_-\) and \(w_-\) gives
\[
\delta\overline{\mathcal S}= \int_0^T \bigg[ \delta\mu_- \left(v_- - \frac{\delta h_-}{\delta \mu}\right) - \left( \dot\mu_- - \ad_{v_-}^* \mu_- + d\left( \mathrm{L}_g \right)^*_e \left(\frac{\delta h_-}{\delta g}\right) \right) (w_-) \bigg] dt.
\]
Since \(\delta\mu_-\) and \(w_-\) are independent, both parentheses must vanish. The first yields \(v_-=\frac{\delta h_-}{\delta \mu}\). Substituting this into the second gives \eqref{eq:LP-left}.
\end{proof}

\begin{proof}[Proof of Lemma~\ref{lem:variation-transported-velocity}]
By the definition of the curvature tensor $R$, for a family of vectors $u_\e\in T_{\gamma(0)}Q$,
\begin{equation*}
  \frac{D}{dt} \frac{D}{d \e}\bigg|_{\e=0} \Gamma(\gamma_\e)^t_0 u_\e - \frac{D}{d \e}\bigg|_{\e=0} \frac{D}{dt} \Gamma(\gamma_\e)^t_0 u_\e = R\left( \gamma'(t), \delta\gamma(t) \right) \Gamma(\gamma)^t_0 u_0.
\end{equation*}
The second term at the left-hand side vanishes, since $\Gamma(\gamma_\e)^t_0 u_\e$ is parallel along the $t$-direction. Recalling the definition of covariant derivative $\frac{D}{dt} = \Gamma_s^t \circ \frac{d}{dt} \circ \Gamma_t^s$, we get
\begin{equation*}
  \frac{d}{dt} \Gamma(\gamma)_t^s \frac{D}{d \e}\bigg|_{\e=0} \Gamma(\gamma_\e)^t_0 u_\e = \Gamma(\gamma)_t^s \left( R\left( \gamma'(t), \delta\gamma(t) \right) \Gamma(\gamma)^t_0 u_0 \right).
\end{equation*}
Thus,
\begin{equation*}
  \frac{D}{d \e}\bigg|_{\e=0} \Gamma(\gamma_\e)^t_0 u_\e - \Gamma(\gamma)_0^t \frac{d}{d \e}\bigg|_{\e=0} u_\e = \int_0^t \Gamma(\gamma)_r^t \left( R\left( \gamma'(r), \delta\gamma(r) \right) \Gamma(\gamma)^r_0 u_0 \right) dr.
\end{equation*}
Now, taking $u_\e$ as $v_\e(t)$, and using \eqref{trans-vel-family} and \eqref{trans-vel-var}, we reach at
\[
\frac{D}{d\epsilon}\bigg|_{\epsilon=0} \gamma'_\e(t) - \Gamma(\gamma)_0^t\,\delta v(t) = \int_0^t \Gamma(\gamma)_r^t \left( R\bigl(\gamma'(r),\delta\gamma(r)\bigr)\,\Gamma(\gamma)_0^r v(t) \right) dr.
\]
Recalling \eqref{eq:variation-der},
the left-hand side of the above equation is precisely \(\delta\dot\gamma(t)\), and we obtain
\[
\Gamma(\gamma)_0^t\,\delta v(t) = \frac{D}{dt}\delta\gamma(t)-T(\gamma'(t),\delta\gamma(t)) - \int_0^t \Gamma(\gamma)_r^t \left(R\bigl(\gamma'(r),\delta\gamma(r)\bigr)\,\Gamma(\gamma)_0^r v(t)\right)dr.
\]
Multiplying both sides by \(\Gamma(\gamma)_t^0\) and using \(\Gamma(\gamma)_t^0\frac{D}{dt}\delta\gamma(t)=\frac{d}{dt}(\Gamma(\gamma)_t^0\delta\gamma(t))\), we arrive at \eqref{eq:delta-v-main}.
\end{proof}

\section{Right formalism}\label{sec-app-right-conn}

This section summarizes the right versions of invariant connections and reduction results, which are obtained by direct analogy with their left-invariant counterparts in Section \ref{sec-inv-conn} and Subsection \ref {subsec-red}. We only list the results here and omit the detailed derivations.

\subsection{Right-invariant connections}\label{sec:right-conn}

By complete analogy, a connection $\nabla$ is called \emph{right-invariant} if right translations are affine isomorphisms. It also induces a bilinear map \(\nabla^{\mathrm{R}}:\mathfrak g\times\mathfrak g\to\mathfrak g\), now called the \emph{right bilinear map}, via right-invariant vector fields:
\[
\nabla^{\mathrm{R}}_v w := (\nabla_{v^{\mathrm R}}w^{\mathrm R})_e,
\qquad \nabla_{v^{\mathrm R}}w^{\mathrm R}=(\nabla^{\mathrm{R}}_v w)^{\mathrm R}.
\]
The results for right-invariant connections are obtained mutatis mutandis from those for left-invariant ones; we therefore list the corresponding formulae without repeating the proofs. 

The covariant derivative of arbitrary right-invariant vector fields \(V\), \(W\) is
\begin{equation*}
(\nabla_V W)_g = \left[ V_g(W_g g^{-1}) + \nabla^{\mathrm{L}}_{V_g g^{-1}} (W_g g^{-1}) \right] g, \quad g\in G.
\end{equation*}
The curve derivative formula for right-invariant connections reads
\begin{equation*}
\frac{D}{dt}W_{g(t)}
=
d(\mathrm{R}_{g(t)})\left( \frac{d}{dt} + \nabla^{\mathrm{R}}_{d(\mathrm{R}_{g(t)^{-1}})g'(t)} \right)
\left( d(\mathrm{R}_{g(t)^{-1}})W_{g(t)} \right),
\end{equation*}
and its dual version is
\begin{equation*}
\frac{D}{dt}\eta_{g(t)}
=
d(\mathrm{R}_{g(t)^{-1}})^*
\left( \frac{d}{dt} - \nabla^{\mathrm{R}*}_{d(\mathrm{R}_{g(t)^{-1}})g'(t)} \right)
\left( d(\mathrm{R}_{g(t)})^*\eta_{g(t)} \right).
\end{equation*}
The corresponding parallel transport operators are
\begin{align*}
\Gamma(g)_s^t
&=
d(\mathrm{R}_{g(t)})\,
\mathcal T\exp\left( -\int_s^t \nabla^{\mathrm{R}}_{d(\mathrm{R}_{g(r)^{-1}})g'(r)}\,dr \right)
\,d(\mathrm{R}_{g(s)^{-1}}), 
\\
\Gamma^*(g)_s^t
&=
d(\mathrm{R}_{g(t)^{-1}})^*\,
\mathcal T\exp\left( \int_s^t \nabla^{\mathrm{R}*}_{d(\mathrm{R}_{g(r)^{-1}})g'(r)}\,dr \right)
\,d(\mathrm{R}_{g(s)})^*. 
\end{align*}

The curvature and torsion of a right-invariant connection exhibit a sign change in the commutator term compared to the left-invariant case:
\begin{gather*}
\begin{aligned}
(R(U,V)W)_g
= d(\mathrm{R}_g)\Big[ &
\nabla^{\mathrm{R}}_{d(\mathrm{R}_{g^{-1}})U_g}
\nabla^{\mathrm{R}}_{d(\mathrm{R}_{g^{-1}})V_g}
\,d(\mathrm{R}_{g^{-1}})W_g - \nabla^{\mathrm{R}}_{d(\mathrm{R}_{g^{-1}})V_g}
\nabla^{\mathrm{R}}_{d(\mathrm{R}_{g^{-1}})U_g}
\,d(\mathrm{R}_{g^{-1}})W_g \\
& + \nabla^{\mathrm{R}}_{[d(\mathrm{R}_{g^{-1}})U_g,\,d(\mathrm{R}_{g^{-1}})V_g]}
\,d(\mathrm{R}_{g^{-1}})W_g \Big],
\end{aligned} 
\\
\begin{aligned}
T(V,W)_g
=
d(\mathrm{R}_g)\Big[ &
\nabla^{\mathrm{R}}_{d(\mathrm{R}_{g^{-1}})V_g}
\,d(\mathrm{R}_{g^{-1}})W_g
- \nabla^{\mathrm{R}}_{d(\mathrm{R}_{g^{-1}})W_g}
\,d(\mathrm{R}_{g^{-1}})V_g + [d(\mathrm{R}_{g^{-1}})V_g,\,d(\mathrm{R}_{g^{-1}})W_g]
\Big].
\end{aligned} 
\end{gather*}
For right-invariant vector fields,
\begin{gather*}
R(u^{\mathrm R},v^{\mathrm R})w^{\mathrm R}
=
\left(
\nabla^{\mathrm{R}}_u\nabla^{\mathrm{R}}_v w
- \nabla^{\mathrm{R}}_v\nabla^{\mathrm{R}}_u w
+ \nabla^{\mathrm{R}}_{[u,v]_{\mathfrak g}}
w
\right)^{\mathrm R}, \\
T(v^{\mathrm R},w^{\mathrm R})
=
\left(
\nabla^{\mathrm{R}}_v w - \nabla^{\mathrm{R}}_w v + [v,w]_{\mathfrak g}
\right)^{\mathrm R}.
\end{gather*}

\begin{remark}
The signs in the curvature and torsion formulae echo the fact that left-invariant vector fields form a Lie algebra isomorphic to \(\mathfrak g\), while right-invariant vector fields form an anti-isomorphic copy, as in \eqref{Lie-bracket-right}.
\end{remark}

The right version of the formula for exchange of derivatives can be proved by \eqref{eq:cov-curve-left} and \eqref{eq:torsion-left}, and is again independent of the choice of right-invariant covariant derivatives:
\begin{equation}\label{eq:exchange-der-right-conn}
\pt_s \left( \pt_t g_s(t) g_s(t)^{-1} \right) - \pt_t \left( \pt_s g_s(t) g_s(t)^{-1} \right) = \left[ \pt_s g_s(t) g_s(t)^{-1}, \pt_t g_s(t) g_s(t)^{-1} \right]_\mathfrak g.
\end{equation}
Equations \eqref{eq:exchange-der-left-conn} and \eqref{eq:exchange-der-right-conn} are two manifestations of the same geometric identity, expressed in terms of left- and right-trivialized velocities, respectively.

\subsection{Right formalism for geometric reductions}

The right-invariant reduction is entirely analogous to the left-invariant case developed in the previous subsection. By replacing left translations with right translations and using the properties of right-invariant connections from Section \ref{sec:right-conn}, one obtains the corresponding reduction equations. 

We trivialize the tangent bundle using right translations. For an element \(g\in G\), define the \emph{right-trivialized velocity}
\begin{equation*}
v_+(t) := \dot g(t) g(t)^{-1} \in \mathfrak g.
\end{equation*}
The Lagrangian is expressed in reduced coordinates as 
$l_+(g,v)=L(g, v g).$
The vertical differential satisfies
\(
d_{\dot g}L = d\left(\mathrm{R}_{g^{-1}}\right)^*_g \frac{\delta l_+}{\delta v}.
\)

For a variation \(g_\epsilon(\cdot)\) with fixed endpoints, define the right-trivialized variation field
\begin{equation*}
w_+(t) = \delta g(t) g(t)^{-1},\qquad w_+(0)=w_+(T)=0.
\end{equation*}
The variation of the reduced velocity follows from the right commutation relation \eqref{eq:exchange-der-right-conn}:
\begin{equation*}
\delta v_+(t) = \frac{\partial}{\partial\epsilon}\bigg|_{\epsilon=0}\bigl(\dot g_\epsilon(t) g_\epsilon(t)^{-1}\bigr)
= \dot w_+(t) - [v_+(t),w_+(t)]_{\mathfrak g}
= \dot w_+(t) - \ad_{v_+(t)} w_+(t).
\end{equation*}

It can be proved (see \cite[Remark 7.8]{HSS09}) that a curve \(g(\cdot)\in C^2([0,T],G)\) is a critical point of the action \(\mathcal S[g(\cdot)]=\int_0^T L(t,g,\dot g)dt\) under fixed-endpoint variations if and only if the reduced variables \((g(\cdot),v_+(\cdot))\) satisfy the \emph{right Euler--Poincar\'e equation}
\begin{equation}\label{eq:EP-right}
\frac{d}{dt}\left(\frac{\delta l_+}{\delta v}\right) + \ad_{v_+(t)}^*\left(\frac{\delta l_+}{\delta v}\right)
= d\left(\mathrm{R}_{g(t)}\right)^*_e\left(\frac{\delta l_+}{\delta g}\right).
\end{equation}


On the Hamiltonian side, define the \emph{right-trivialized momentum}
\begin{equation*}
\mu_+ := d\left(\mathrm{R}_g\right)^*_e p \in \mathfrak g^*.
\end{equation*}
The reduced Hamiltonian is \(h_+(g,\mu)=H(g,\mu g)\) and satisfies \(\pt_p H = d\left(\mathrm{R}_g\right)_e\frac{\delta h_+}{\delta \mu}\).
Then a curve \((g(\cdot),p(\cdot))\in C^2([0,T],T^*G)\) is a critical point of the phase-space action \(\overline{\mathcal S}[g(\cdot), p(\cdot)]=\int_0^T\bigl[p(t)(\dot g(t))-H(g(t),p(t),t)\bigr]dt\) under fixed-configuration-endpoint variations if and only if the reduced variables \((g(\cdot),\mu_+(\cdot))\) satisfy the \emph{right Lie--Poisson equation} (cf. \cite[Remark 9.5]{HSS09})
\begin{equation}\label{eq:LP-right}
\dot\mu_+(t) + \ad_{\frac{\delta h_+}{\delta \mu}}^* \mu_+(t) = - d\left(\mathrm{R}_{g(t)}\right)^*_e\left(\frac{\delta h_+}{\delta g}\right).
\end{equation}

The reduced Legendre transform provides the link between the Lagrangian and Hamiltonian descriptions:
\[
\mu_+ = \frac{\delta l_+}{\delta v},\qquad v_+ = \frac{\delta h_+}{\delta \mu}, \qquad \mu_+ v_+ = h_+ + l_+.
\]


\begin{theorem}[Right Euler--Poincar\'e and Lie--Poisson reductions and reconstructions]
Let \(G\) be equipped with a right-invariant connection whose associated right bilinear map is \(\nabla^{\mathrm{R}}\). The Euler--Lagrange equation \eqref{EL-Lie} is equivalent to the right Euler--Poincar\'e equation \eqref{eq:EP-right} via the reconstruction equation \(g'(t) = v_+(t) g(t)\). Hamilton's equations \eqref{Hamilton-Lie} are equivalent to the right Lie--Poisson equation \eqref{eq:LP-right}, via the reconstruction equations \(g'(t) = \big( \frac{\delta h_+}{\delta \mu} \big) g(t)\) and \(p(t) = d\left( \mathrm{R}_{g(t)^{-1}} \right)^*_{g(t)} \mu_+(t)\).
\end{theorem}



\begin{remark}
The Euler--Poincar\'e equations \eqref{eq:EP-left}, \eqref{eq:EP-right}, and the Lie--Poisson equations \eqref{eq:LP-left}, \eqref{eq:LP-right} are the standard forms appearing in the mechanics of ideal fluids and rigid bodies, independent of the choice of left- or right-invariant connections. 
The reduction results are therefore universal: they hold for any invariant connection, and the Cartan--Schouten connections considered in later sections are but one particularly bi-invariant family of such connections.
\end{remark}

\section{Local form of the integro-differential Euler--Lagrange equation}\label{sec-app-eqn}

To make equation \eqref{eq:int-diff-EL} more explicit, we introduce a parallel frame along \(\gamma\). Let \((e_i(t))\) be a frame such that \(\nabla_{\gamma'}e_i=0\) and \(e_i(0)\) is a fixed orthonormal basis of \(T_{\gamma(0)}Q\). Write
\[
\gamma'(t)=v^i(t)e_i(t),\qquad \delta\gamma(t)=w^i(t)e_i(t).
\]
In this frame, the torsion and curvature components are denoted as
\[
T(e_i(t),e_j(t))=T_{ij}^k(t)e_k(t),\qquad
R(e_i(t),e_j(t))e_k(t)=R_{ijk}^{l}(t)e_l(t).
\]
Since \(e_i(t)=\Gamma(\gamma)_0^t e_i(0)\), we have
\[
v(t)=\Gamma(\gamma)_t^0\gamma'(t)=v^i(t)e_i(0),\qquad
\Gamma(\gamma)_t^0\delta\gamma(t)=w^i(t)e_i(0).
\]
Substituting into \eqref{eq:delta-v-main} gives the components
\[
\delta v^i(t)=\dot w^i(t)-v^j(t)w^k(t)T_{jk}^i(t)-v^n(t)\int_0^t v^j(r)w^k(r)R_{jkn}^{i}(r)\,dr.
\]
Define
\[
\delta_{q^i}l(t):=\frac{\delta l}{\delta q}(t)(e_i(t)),\qquad
\delta_{v^i}l(t):=\frac{\delta l}{\delta v}(t)(e_i(0)).
\]
Then the variation of the action becomes
\[
\delta\mathcal S=\int_0^T\bigl[w^i(t)\delta_{q^i}l(t)+\delta v^i(t)\delta_{v^i}l(t)\bigr]dt.
\]
After inserting the expression for \(\delta v^i(t)\), integrating by parts the \(\dot w^i\) term, and swapping the order of integration in the curvature double integral, we obtain
\[
\delta\mathcal S=\int_0^T w^i(t)\Bigl[
\delta_{q^i}l(t)-\frac{d}{dt}\delta_{v^i}l(t)
+v^j(t)T_{ij}^k(t)\delta_{v^k}l(t)
+v^j(t)R_{ijk}^{n}(t)\int_t^T v^k(r)\delta_{v^n}l(r)\,dr
\Bigr]dt.
\]
Since the \(w^i(t)\) are arbitrary, we obtain the local form.

\begin{corollary}[Local integro-differential Euler--Lagrange equation]
In a parallel frame, the Euler--Lagrange equation \eqref{eq:int-diff-EL} takes the component form
\begin{equation*}
\delta_{q^i}l(t)=\frac{d}{dt}\delta_{v^i}l(t)
-v^j(t)T_{ij}^k(t)\delta_{v^k}l(t)
-v^j(t)R_{ijk}^{n}(t)\int_t^T v^k(r)\delta_{v^n}l(r)\,dr.
\end{equation*}
\end{corollary}

\section{Nilpotent case: the three-dimensional Heisenberg group}\label{sec-app-H3}

The observation in Remark \ref{rem-simple} makes nilpotent Lie groups a natural class of examples where the connection-dependent variational principle can be studied in full detail. We illustrate the general theory on the three-dimensional Heisenberg group \(\mathrm{H}_3\), the simplest non-abelian nilpotent Lie group, and show that the parallel transport operator can be expressed explicitly.


Let us adopt the (non-symmetric) vector representation
\[
g=\varg(a,b,c)=\begin{pmatrix}
1 & a & c\\
0 & 1 & b\\
0 & 0 & 1
\end{pmatrix},\qquad (a,b,c)\in\mathbb R^3.
\]
This representation amounts to adopting $(a,b,c)$ as coordinates of \(\mathrm{H}_3\).
The group multiplication is given by
\(
\varg(a,b,c)\,\varg(a',b',c')=\varg\bigl(a+a',\; b+b',\; c+c'+a b'\bigr).
\)
The identity is \(e=(0,0,0)\). The inverse of $\varg(a,b,c)$ is
\(
  \varg(a,b,c)^{-1} = \varg(-a, -b, a b-c).
\)

\subsubsection*{Heisenberg algebra}

The Lie algebra \(\mathfrak h_3\) corresponding to the left-invariant vector fields at the identity has a basis \(\{E_1,E_2,E_3\}\), given by
\[
E_1=\frac{\partial}{\partial a}\bigg|_e = \begin{pmatrix}
0 & 1 & 0\\
0 & 0 & 0\\
0 & 0 & 0
\end{pmatrix},
\qquad
E_2=\frac{\partial}{\partial b}\bigg|_e = \begin{pmatrix}
0 & 0 & 0\\
0 & 0 & 1\\
0 & 0 & 0
\end{pmatrix},
\qquad
E_3=\frac{\partial}{\partial c}\bigg|_e = \begin{pmatrix}
0 & 0 & 1\\
0 & 0 & 0\\
0 & 0 & 0
\end{pmatrix}.
\]
They satisfy the commutator relation
\(
[E_1,E_2]=E_3, [E_1,E_3]=0, [E_2,E_3]=0,
\)
and \(E_3\) is central. Hence, the Heisenberg group is two-step nilpotent.
In this coordinate system, we identify 
\(
v = v^1 E_1 + v^2 E_2 + v^3 E_3 \in \mathfrak h_3
\)
with
\(
\big(v^1,v^2,v^3\big)^\top \in \R^3.
\)
Then the adjoint action is
\begin{equation}\label{ad-Heis}
  \ad_{(u^1,u^2,u^3)^\top}\big(v^1,v^2,v^3\big)^\top = \big(0,0, u^1 v^2 - u^2 v^1\big)^\top.
\end{equation}

The tangent map of the left translation \(\mathrm{L}_g\) by \(g=\varg(a,b,c)\) at the identity is
\begin{equation}\label{left-Heis}
  d(\mathrm{L}_g)_e\big(v^1,v^2,v^3\big)^\top 
  = \big(v^1, v^2, v^3+a v^2\big)^\top = v^1\, \pt_a|_g + v^2\, \pt_b|_g + (v^3+a v^2)\, \pt_c|_g.
\end{equation}
For a curve \(g(\cdot)=(a(\cdot),b(\cdot),c(\cdot))\),
\begin{equation}\label{reduced-v-Heis}
g(t)^{-1}g'(t)= \big(\dot a(t), \dot b(t), \dot c(t)-a(t)\dot b(t)\big)^\top.
\end{equation}

\subsubsection*{Parallel transport}

Let \(g(t)=(a(t),b(t),c(t))\). 
Since the Heisenberg group is two-step nilpotent, the parallel transport operator reduces to \eqref{para-trans-heis}. 
A direct computation using \eqref{ad-Heis}, \eqref{left-Heis} and \eqref{reduced-v-Heis} yields
\begin{equation}\label{para-trans-Heis}
\begin{aligned}
  \Gamma(g)_s^t \big(v^1,v^2,v^3\big)^\top &= \bigg( v^1, v^2, v^3+(a(t)-a(s))v^2 - \lambda\int_s^t \bigl(\dot a(r)v^2-\dot b(r)v^1\bigr)\,dr \bigg)^\top \\
  &= \big( v^1, v^2, v^3+ (1-\lambda) (a(t)-a(s))v^2+ \lambda (b(t)-b(s))v^1 \big)^\top.
\end{aligned}
\end{equation}
This formula is explicit and involves only an ordinary integral of the velocity components.

\subsubsection*{Duality}

Let \(\{\theta^1,\theta^2,\theta^3\}\) be the basis of \(\mathfrak h_3^*\) dual to \(\{E_1,E_2,E_3\}\), i.e. 
\(
\theta^1=da|_e, \quad \theta^2=db|_e, \quad \theta^3=dc|_e.
\)
We identify an element \(\eta\in\mathfrak h_3^*\) with a row vector \((\eta_1,\eta_2,\eta_3)\) via \(\eta=\eta_1\theta^1+\eta_2\theta^2+\eta_3\theta^3\).
The dual of the left translation \eqref{left-Heis} \(d(\mathrm{L}_g)_e^*:\mathfrak h_3^*\to\mathfrak h_3^*\) is given by
\begin{equation*}
  d(\mathrm{L}_g)_e^*\eta = \eta\circ d(\mathrm{L}_g)_e = (\eta_1, \eta_2+a \eta_3, \eta_3) = \eta_1\, da|_g + (\eta_2+a \eta_3)\, db|_g + \eta_3\, dc|_g.
\end{equation*}
The dual operator of the adjoint \eqref{ad-Heis} \(\ad_u^*:\mathfrak h_3^*\to\mathfrak h_3^*\), is defined by \((\ad_u^*\eta)(v)=\eta(\ad_u v)\), and computed as
\begin{equation}\label{ad-dual-Heis}
\ad_u^*\eta = \big(-\eta_3 u^2, \eta_3 u^1, 0\big).
\end{equation}

The dual parallel transport on covectors, using the explicit forms \eqref{para-trans-Heis}, is given by
\begin{equation}\label{para-trans-dual-Heis}
\Gamma^*(g)_t^s \eta
=
\big(
\eta_1
+ \lambda\eta_3 (b(t)-b(s)),
\eta_2 + (1-\lambda)\eta_3 (a(t)-a(s)),
\eta_3
\big),
\end{equation}

\subsubsection*{The reduced system}

We now apply the integro-differential Euler--Poincar\'e equation \eqref{eq:reduced-int-diff-EP} to the Heisenberg group. Since the Heisenberg group is two-step nilpotent, the curvature of the Cartan--Schouten connection vanishes identically: indeed, for any \(u,v,w\in\mathfrak h_3\), we have \(R(u,v)w=\lambda(\lambda-1)[[u,v]_{\mathfrak g},w]_{\mathfrak g}=0\) because \([u,v]_{\mathfrak g}\) lies in the center. Consequently, the curvature integral term in \eqref{eq:reduced-int-diff-EP} drops out, and the equation reduces to
\[
\Gamma(g)_t^0 \frac{\delta l}{\delta g}
=
\frac{d}{dt}\frac{\delta l}{\delta v}
+(2\lambda-1)\ad_{v(t)}^*\frac{\delta l}{\delta v},
\]
together with the reconstruction equation \(\dot g(t)=\Gamma(g)_0^t v(t)\).

Let us write the equation in components with respect to the dual basis \(\{\theta^1,\theta^2,\theta^3\}\) of \(\mathfrak h_3^*\). First, we have
\(
\frac{\delta l}{\delta v} = \frac{\partial l}{\partial v^i} \theta^i = \left( \frac{\partial l}{\partial v^1}, \frac{\partial l}{\partial v^2}, \frac{\partial l}{\partial v^3} \right).
\)
The coadjoint action is given by \eqref{ad-dual-Heis}:
\[
\ad_v^* \frac{\delta l}{\delta v} = \left(-\frac{\partial l}{\partial v^3} v^2, \frac{\partial l}{\partial v^3} v^1, 0\right).
\]

The left-hand side requires the parallel transport of the covector \(\frac{\delta l}{\delta g}\) from \(g(t)\) back to \(g(0)\). In the coframe $\{da, db, dc\}$, \(\frac{\delta l}{\delta g}\) has the expression
\(
\frac{\delta l}{\delta g} = \frac{\partial l}{\partial a} da + \frac{\partial l}{\partial b} db + \frac{\partial l}{\partial c} dc = \left( \frac{\partial l}{\partial a}, \frac{\partial l}{\partial b}, \frac{\partial l}{\partial c} \right).
\)
Let \(g(t)=(a(t),b(t),c(t))\). The dual parallel transport from \(t\) to \(0\) is obtained from \eqref{para-trans-dual-Heis}:
\[
\Gamma(g)_t^0 \frac{\delta l}{\delta g}
=
\left(
\frac{\partial l}{\partial a}
+ \lambda\frac{\partial l}{\partial c} (b(t)-b_0),
\frac{\partial l}{\partial b} + (1-\lambda)\frac{\partial l}{\partial c} (a(t)-a_0),
\frac{\partial l}{\partial c}
\right).
\]
Similarly, using \eqref{para-trans-Heis} the reconstruction equation yields
\begin{equation*}
  \dot g(t)=\Gamma(g)_0^t v(t) = \big( v^1(t), v^2(t), v^3(t)+ (1-\lambda) (a(t)-a_0)v^2(t)+ \lambda (b(t)-b_0)v^1(t) \big)^\top.
\end{equation*}

Therefore, the reduced Euler--Poincar\'e equation becomes the following system:
\begin{equation*}
\left\{
\begin{aligned}
& \frac{d}{dt}\frac{\partial l}{\partial v^1} - (2\lambda-1) \frac{\partial l}{\partial v^3} v^2(t) = \frac{\partial l}{\partial a} + \lambda\frac{\partial l}{\partial c} (b(t)-b_0), \\
& \frac{d}{dt}\frac{\partial l}{\partial v^2} + (2\lambda-1) \frac{\partial l}{\partial v^3} v^1(t) = \frac{\partial l}{\partial b} + (1-\lambda)\frac{\partial l}{\partial c} (a(t)-a_0), \\
& \frac{\partial l}{\partial c} = \frac{d}{dt}\frac{\partial l}{\partial v^3}, \\
& \dot a(t) = v^1(t), \\
& \dot b(t) = v^2(t), \\
& \dot c(t) = v^3(t)+ (1-\lambda) (a(t)-a_0)v^2(t)+ \lambda (b(t)-b_0)v^1(t),
\end{aligned}
\right.
\end{equation*}
with boundary conditions $(a(0),b(0),c(0)) = (a_0,b_0,c_0)$, $(a(T),b(T),c(T)) = (a_T,b_T,c_T)$.

Finally, we note that although the curvature vanishes, the parallel transport still contains the \(\lambda\)-dependent integral terms, which affect the left-hand side through the boundary values \(a_0,b_0\). These terms arise from the non-commutativity of the group and are essential even in the flat case.

\subsubsection*{Numerical illustration}

We conclude this example with a numerical simulation of the reduced Euler--Poincar\'e equations on the Heisenberg group for the harmonic oscillator Lagrangian
\[
l = \frac12 \big[ (v^1)^2+(v^2)^2+(v^3)^2 \big] - \frac12 (a^2+b^2+c^2).
\]
The boundary value problem consists of the ODE system
\[
\left\{
\begin{aligned}
\dot a &= v^1,\\
\dot b &= v^2,\\
\dot c &= v^3 + (1-\lambda)av^2 + \lambda bv^1,\\
\dot v^1 &= - a - \lambda cb + (2\lambda-1)v^2 v^3,\\
\dot v^2 &= - b - (1-\lambda)ca - (2\lambda-1)v^1 v^3,\\
\dot v^3 &= - c.
\end{aligned}
\right.
\]
with fixed initial position at the identity \((a(0),b(0),c(0))=(0,0,0)\) and prescribed final position \((a(T),b(T),c(T))\). The velocities at the endpoints are free. We solve this two-point boundary value problem numerically using a standard collocation method (implemented in SciPy's \texttt{solve\_bvp}) for three values of the Cartan--Schouten parameter: \(\lambda=0\), \(\lambda=\tfrac12\), and \(\lambda=1\). The final time is set to \(T=1\).

Figure~\ref{fig:Heis-trajectories} shows the computed trajectories in the \((a,b,c)\)-space for several representative endpoint conditions. In each panel, the three curves correspond to the three values of \(\lambda\). The starting point is marked with a circle and the target endpoint with a square.

\begin{figure}[htbp]
\centering
\includegraphics[width=0.45\textwidth]{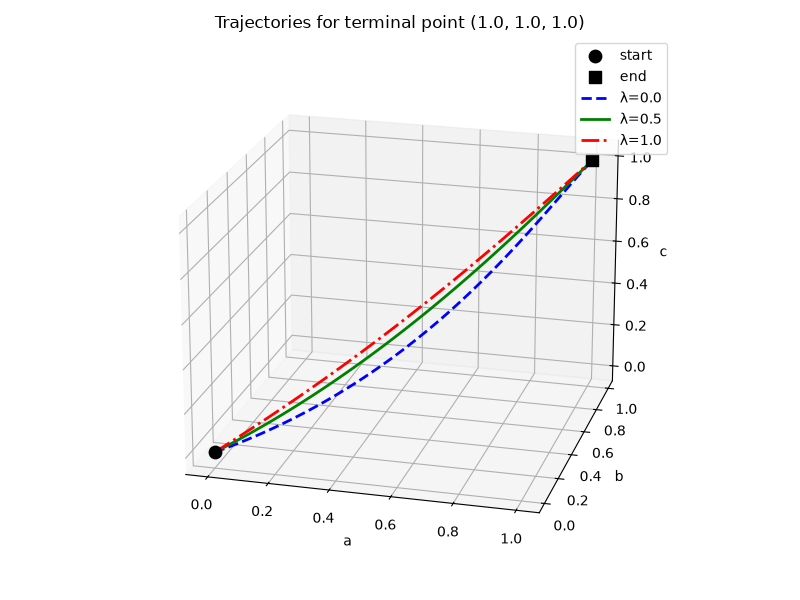}
\includegraphics[width=0.45\textwidth]{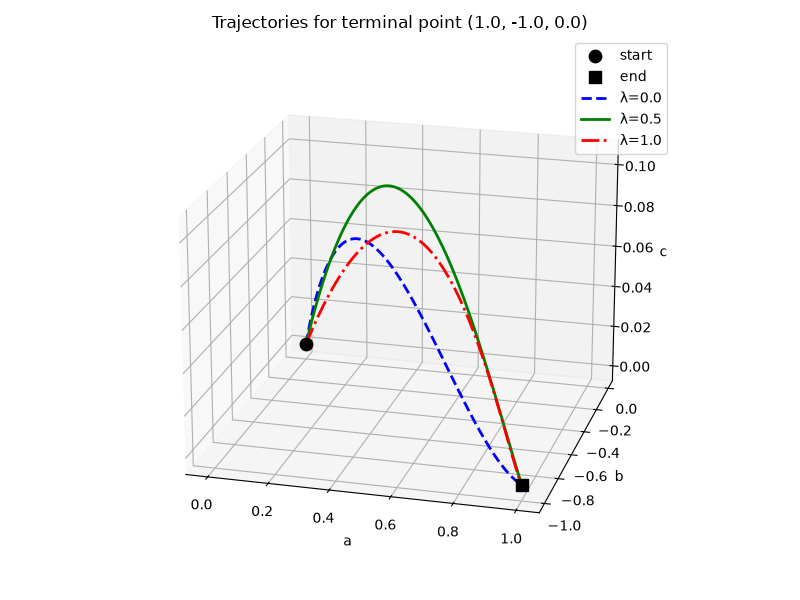}
\caption{Trajectories on the Heisenberg group for the harmonic oscillator Lagrangian, for different endpoint conditions. Dashed blue curves: \(\lambda=0\); solid green: \(\lambda=\tfrac12\); dash-dotted red: \(\lambda=1\). The starting point is at the origin and the target endpoint is shown as a square.}
\label{fig:Heis-trajectories}
\end{figure}

The numerical results clearly indicate that the trajectory for \(\lambda=\tfrac12\) lies between those for \(\lambda=0\) and \(\lambda=1\), confirming that the Cartan--Schouten parameter acts as an interpolation parameter between the $(\pm)$-connections. This behaviour is consistent with the fact that the parallel transport and the reduced equations depend continuously on \(\lambda\), and the extreme values \(\lambda=0\) and \(\lambda=1\) correspond to purely left and purely right trivializations of the group, respectively. The intermediate values of \(\lambda\) provide a continuous family of connections that interpolate between these two geometrically distinct cases.

This example demonstrates the practical utility of the connection-dependent variational principle: even on a simple non-abelian group, the parameter \(\lambda\) has a tangible effect on the dynamics, and the numerical solution of the resulting boundary value problem is feasible with standard methods.

\section{Non-nilpotent case: the rigid-body}\label{sec-app-rigid-body}

This appendix records the Lie-algebra conventions, the harmonic-oscillator model, and the discrete equations used for the computations in Subsection~\ref{subsec-rigid-body}.

\subsection{The Lie algebra \texorpdfstring{\(\mathfrak{so}(3)\)}{so(3)} and its dual}

Let \(\mathfrak{so}(3)\) be the Lie algebra of \(\mathrm{SO}(3)\), consisting of \(3\times 3\) skew-symmetric matrices. Take the basis
\[
E_1=\begin{pmatrix}0&0&0\\0&0&-1\\0&1&0\end{pmatrix},\quad
E_2=\begin{pmatrix}0&0&1\\0&0&0\\-1&0&0\end{pmatrix},\quad
E_3=\begin{pmatrix}0&-1&0\\1&0&0\\0&0&0\end{pmatrix}.
\]
These generators satisfy the commutation relations
\(
[E_1,E_2]=E_3, [E_2,E_3]=E_1, [E_3,E_1]=E_2.
\)

We recall the standard identifications for \(\mathfrak{so}(3)\) and its dual, following \cite[Examples 5.35, 5.36]{HSS09}. The hat map \(\widehat{\cdot}:\mathbb R^3\to \mathfrak{so}(3)\) is defined by
\(\widehat{v} = v^1 E_1 + v^2 E_2 + v^3 E_3, v=(v^1,v^2,v^3)\in\mathbb R^3,\)
so that \(\widehat{v}\,w = v\times w\) and \(\ad_{\widehat{v}} \widehat{w} = [\widehat{v},\widehat{w}] = \widehat{v\times w} = \widehat{\widehat{v}\,w}\).
The inverse of the hat map is denoted by \(\widecheck{\cdot}:\mathfrak{so}(3)\to\mathbb R^3\).
The dual Lie algebra \(\mathfrak{so}(3)^*\) is identified with \(\mathbb R^3\) via the breve map \(\breve{\cdot}:\mathbb R^3\to \mathfrak{so}(3)^*\), defined by
\(
\breve{\mu} (\widehat{v}) = \mu\cdot v, \mu,v\in\mathbb R^3.
\)
Under this identification, the coadjoint action of \(\mathfrak{so}(3)\) on \(\mathfrak{so}(3)^*\) is given by
\[
\ad_{\widehat{v}}^* \breve{\mu} = - \widebreve{v\times \mu},
\]
because \(\ad_{\widehat{v}}^* \breve{\mu} (\widehat{w}) = \breve{\mu} (\ad_{\widehat{v}} \widehat{w}) = \mu\cdot (v\times w) = (\mu\times v)\cdot w\).
We shall consistently use this identification and denote elements of \(\mathfrak{so}(3)^*\) by vectors \(\mu, \xi, \eta\in\mathbb R^3\).

\subsection{The harmonic oscillator}

We now add a potential that is quadratic in the deviation from the identity. Let
\[
l(g,v) = \frac12\, v\cdot \mathbb I v - V(g),
\qquad V(g) = \frac12 (3 - \operatorname{tr}g).
\]
For small rotations, \(V(g)\approx \frac12\theta^2\), where \(\theta\) is the rotation angle, so this is the natural harmonic oscillator potential on \(\mathrm{SO}(3)\).

The derivative of \(l\) with respect to \(g\) is the covector \(\frac{\partial l}{\partial g} = -\frac12 \operatorname{tr}'(g)\), where \(\operatorname{tr}'(g)\) is the derivative of the trace operator at \(g\). Suppose that this covector is represented by \(\zeta\in\mathbb R^3\) (via the breve map), that is,
\begin{equation*}
  d\big(\mathrm{L}_{g^{-1}}\big)^* \left(\frac{\delta l}{\delta g}\right) = \breve{\zeta} \in \mathfrak{so}(3)^*.
\end{equation*}
Then
\(
\zeta\cdot v = -\frac12 \operatorname{tr}(g\,\widehat{v}), v\in\mathbb R^3.
\)
Since \(\operatorname{tr}(g\,\widehat{v}) = -2\,\widecheck{\operatorname{Skew}(g)}\cdot v\), where \(\operatorname{Skew}(g)=\frac{g-g^\top}{2}\), we find
\[
\zeta = \widecheck{\operatorname{Skew}(g)}.
\]
Thus, the vector representing the left-invariant gradient of the potential is simply the skew-symmetric part of \(g\).

The reduced Euler--Poincar\'e equation \eqref{eq:reduced-int-diff-EP} now becomes
\[
\Gamma^*(g)_t^0 \bigl(\operatorname{Skew}(g)\widecheck{}\bigr)
=
\mathbb I \dot v - (2\lambda-1)\, v\times (\mathbb I v) - \lambda(1-\lambda)\, v\times \xi,
\]
with
\(
\dot\xi = - v\times (\mathbb I v), \xi(T)=0,
\)
and the reconstruction equation
\(
\dot g = \Gamma(g)_0^t v, g(0)=g_0, g(T)=g_T.
\)

\subsection{Discrete systems}

We discretize the interval \([0,T]\) into \(N\) steps of size \(h\). Let \(g_k\approx g(t_k)\), \(v_k\approx v(t_k)\), \(\mu_k = \mathbb I v_k\in\mathbb R^3\), and \(\xi_k\in\mathbb R^3\). The reconstruction step is
\[
g_{k+1} = g_k \exp(h\, \widehat{u_k}),\qquad u_k = P_k \widecheck{g_0^{-1} v_k},
\]
where the matrix exponential \(\exp(\widehat{w})\) is computed via Rodrigues' formula (e.g., \cite[Eq.~(9.2.8)]{MR99}).

The pullback of the ordered exponential \(P_k\) to $\R^3$, still denoted as \(P_k\), is updated by
\[
P_{k+1}=\exp\left(-\lambda h\,\widehat{u_k}\right) P_k,\qquad P_0 = I,
\]
since
\(
\left( \exp\left(-\lambda h\,\ad_{\widehat{u_k}}\right) \widehat{v} \right)\widecheck{}=\exp\left(-\lambda h\,\widehat{u_k}\right) v, v\in \mathfrak g.
\)

The momentum equation is discretized implicitly:
\[
\eta_k = \frac{\mu_{k+1}-\mu_k}{h} - (2\lambda-1)\, v_k\times \mu_k - \lambda(1-\lambda)\, v_k\times \xi_k,
\]
where \(\eta_k=0\) for the free rigid body and
\(
\eta_k = d\big(\mathrm{L}_{g_0}\big)^* \, P_k^{-1}\,\widecheck{\operatorname{Skew}(g_k)}
\)
for the harmonic oscillator.
The auxiliary variables satisfy
\(
\xi_k = \xi_{k+1} - h\, v_k\times \mu_k, \xi_N=0.
\)
The Legendre transform yields \(\mu_k = \mathbb I v_k\).

\section{More about the two-level scheme}\label{sec-app-scheme}

\subsection{Algorithm}

The implementation of the two-level scheme in Subsection \ref{subsec-num-sch} can be organized as follows (pseudocode):
\begin{algorithm}
\caption{Two-level scheme for the integro-differential Euler--Poincar\'e equation}
\KwIn{\(g_0, g_T \in G\), \(\lambda \in [0,1]\), \(N \in \mathbb N\), \(\text{tol}>0\), \(\text{maxIter}>0\)}
\KwOut{Discrete solution \(\{g_k, v_k\}_{k=0}^{N}\)}
Set \(h = T/N\), \(P_0 = I\), choose initial guess \(v_0 \in \mathfrak g\), set \(\xi_k^{(0)}=0\) for all \(k\)\;
\For{\(m = 0\) \KwTo maxIter}{
  \tcc{Outer fixed-point iteration for \(\{\xi_k\}\)}
  \Repeat{\(\|g_N - g_T\| < \text{tol}\)}{
    \tcc{Inner shooting loop for \(v_0\)}
    Set \(g_0\) given, \(P_0 = I\), \(\mu_0 = \frac{\delta l}{\delta v}(t_0,g_0, v_0)\)\;
    \For{\(k = 0\) \KwTo \(N-1\)}{
      \tcc{Forward sweep}
      \(u_k = P_k v_k\)\;
      \(g_{k+1} = g_k \exp(h\,u_k)\)\;
      \(P_{k+1} = \exp(-\lambda h\,\ad_{u_k}) P_k\)\;
      \(\eta_k = d\big(\mathrm{L}_{g_0}\big)^* P_k^* d\big(\mathrm{L}_{g_k^{-1}}\big)^* \left(\frac{\delta l}{\delta g}(t_k,g_k, v_k)\right)\)\;
      \(\mu_{k+1} = \mu_k + h\left( \eta_k - (2\lambda-1)\ad_{v_k}^*\mu_k - \lambda(1-\lambda)\ad_{v_k}^*\xi_k^{(m)} \right)\)\;
      \(v_{k+1} = \left(\frac{\delta l}{\delta v}(t_{k+1},g_{k+1},\cdot)\right)^{-1}(\mu_{k+1})\)\;
    }
    Compute Jacobian \(J = \frac{\partial g_N}{\partial v_0}\) by finite differences\;
    Update \(v_0 \leftarrow v_0 - J^{-1}(g_N - g_T)\)\;
  }
  \tcc{Backward integration for \(\{\xi_k\}\)}
  Set \(\xi_N = 0\)\;
  \For{\(k = N-1\) \KwTo \(0\)}{
    \(\xi_k^{(m+1)} = \xi_{k+1}^{(m+1)} + h\,\ad_{v_k}^*\mu_k\)\;
  }
  \If{\(\max_k \|\xi_k^{(m+1)} - \xi_k^{(m)}\| < \text{tol}\)}{
    \KwRet \(\{g_k, v_k, \mu_k, \xi_k^{(m+1)}\}_{k=0}^{N}\)\;
  }
}
\KwRet $\{g_k, v_k\}_{k=0}^{N}$.
\end{algorithm}

\subsection{Accuracy and improvement}

The scheme described above is \emph{first‑order accurate} in the step size \(h\): the local error in the group element is \(O(h^2)\), and the global error is \(O(h)\). This is simple, robust, and sufficient for many practical applications, especially when the dynamics are not too stiff.

To achieve higher accuracy, one can:
\begin{itemize}
\item Use a higher‑order Magnus expansion for the parallel transport. A second‑order Magnus method on the interval \([t_k,t_{k+1}]\) approximates the path‑ordered exponential by
\[
P_{k+1} = \exp\left( \Omega_1 + \Omega_2 \right) P_k,
\]
where
\[
\Omega_1 = -\lambda h\,\ad_{u_k},\qquad
\Omega_2 = -\frac{\lambda^2 h^2}{12}\left( [\ad_{u_k}, \ad_{u_{k+1}}] \right),
\]
requiring a predictor step for \(u_{k+1}\). This yields second‑order accuracy in \(h\).


\item Use a higher‑order integrator for the momentum equation: Replace the implicit Euler step by a trapezoidal rule or a symplectic integrator, which would improve the conservation properties.
\end{itemize}

\end{appendices}

{\footnotesize
\bibliographystyle{plain}
\bibliography{refs}
}

\end{document}